\documentclass[11 pt]{amsart}
\usepackage[margin=1.0in]{geometry}

\usepackage[T1]{fontenc}
\usepackage[utf8]{inputenc}
\usepackage{lmodern}
\usepackage{hyperref, amsmath, amssymb, mathtools}
\usepackage{verbatim}
\usepackage{amsthm}
\usepackage{amsfonts}
\usepackage{amsmath}
\usepackage{mathrsfs}  
\usepackage{comment}
\usepackage{chngcntr}
\usepackage{blkarray}
\usepackage{tikz-cd}
\usepackage{ctable}

\newtheorem{thm}{Theorem}
\newtheorem*{thm*}{Theorem}
\newtheorem{lemma}[thm]{Lemma}

\newtheorem{conjecture}[thm]{Conjecture}

\newtheorem*{prop*}{Proposition}

\theoremstyle{definition}

\theoremstyle{remark}
\newtheorem{remark}[thm]{Remark}
\numberwithin{thm}{section}
\numberwithin{equation}{section}

\DeclareMathOperator*{\E}{\mathbb{E}}

\newcommand{\Z}{\mathbb{Z}}
\newcommand{\F}{\mathbb{F}}
\newcommand{\N}{\mathbb{N}}
\newcommand{\C}{\mathbb{C}}

\newcommand{\calP}{\mathcal{P}}

\newcommand{\fS}{\mathfrak{S}}
\newcommand{\scrP}{\mathscr{P}}

\DeclareMathOperator{\siegel}{Siegel}
\DeclareMathOperator{\cramer}{Cram\'er}
\DeclareMathOperator{\vol}{Vol}

\usepackage{xcolor}

\newcommand{\Comm}[1]{{\color{olive}{$\blacktriangleright$Comment: #1}}}

\title[The distribution of prime values of random polynomials]{The distribution of prime values of random polynomials}

\author[Kravitz]{Noah Kravitz}
\address{St John's College, Oxford and Mathematical Institute, University of Oxford; St Giles', Oxford OX1 3JP, UK}
\email{noah.kravitz@maths.ox.ac.uk}

\author[Woo]{Katharine Woo}
\address{Stanford University, 450 Serra Mall, Stanford, CA 94306}
\email{katywoo@gmail.com}

\author[Xu]{Max Wenqiang Xu} 
\address{Courant Institute of Mathematical Sciences, 251 Mercer Street, New York 10012, USA}
\email{maxxu1729@gmail.com}

\begin{document}
\maketitle

\begin{abstract}
The Bateman--Horn Conjecture predicts how often an irreducible polynomial $f(x) \in \mathbb{Z}[x]$ assumes prime values. We demonstrate that with sufficient averaging in the coefficients of $f$ (viz. exponential in the size of the inputs), one can not only prove Bateman--Horn results on average but also pin down precise information about the distribution of prime values.  We show that 100\% of polynomials (in an $L^k$ sense for all $k \in \mathbb{N}$) satisfy the Bateman--Horn Conjecture, and that that 100\% of polynomials (in an $L^2$ sense) satisfy an appropriate polynomial analogue of the Hardy--Littlewood Prime Tuples Conjecture.  We use the latter to prove that 100\% of polynomials satisfy the appropriate analogue of the Poisson Tail Conjecture, in the sense that the distribution of the gaps between consecutive prime values around the average spacing is Poisson.

We also study the frequencies of sign patterns of the Liouville function evaluated at the consecutive outputs of $f$; viewing $f$ as a random variable, we establish the limiting distribution for every sign pattern. The Chowla problem along random polynomials is a special case. 
A key input behind all of our arguments is Leng's recent quantitative work on the higher-order Fourier uniformity of the von Mangoldt and M\"obius functions (in turn relying on Leng, Sah, and Sawhney's quantitative inverse theorem for the Gowers norms).
\end{abstract}

\section{Introduction}

\subsection{Prime values along (random) polynomials}
A notorious number-theoretic problem is determining whether a given polynomial $f\in \mathbb{Z}[x]$ assumes prime values at infinitely many positive integer inputs. It is clear that $f$ fails to achieve prime values if it factors, if it is uniformly zero modulo some prime (e.g., $x^2+x+2$ is always even), or if it is eventually always negative.  In 1854, Bunyakovsky \cite{bunjakovskij1854diviseurs} conjectured that these are the only three things that can go wrong: If an irreducible polynomial has positive leading coefficient and its outputs have no common factor, then it should be prime for infinitely many integer inputs.  Dirichlet's Theorem on primes in arithmetic progressions proves this conjecture for linear polynomials.  There is no non-linear polynomial for which the conclusion of Bunyakovsky's Conjecture is known to hold.

In 1962, Bateman and Horn \cite{BatemanHorn} refined Bunyakovsky's Conjecture by  predicting how frequently such prime values should occur. Using quasirandomness heuristics for the distribution of primes, they predicted that if $f(x)\in \Z[x]$ is an irreducible polynomial, then 
\begin{equation}\label{eq: Bateman Horn statement}
    \sum_{n\leq X} \Lambda(f(n)) \sim \fS_f X,
\end{equation}
where $\Lambda$ denotes the generalized von Mangoldt function 
\begin{equation*}
    \Lambda(n) := \begin{cases}
        \log(p), & \text{if } n = \pm p^k \text{ for some prime }p;\\ 
        0, & \text{ otherwise}
    \end{cases}
\end{equation*}
and $\fS_f$ denotes the singular series 
\begin{equation*}
    \fS_f := \prod_{p} \left(\frac{p^{-1} \#\{x\in \F_p: f(x) \in \F_p^\times\}}{1-1/p}\right).
\end{equation*}
Using the same heuristics, they also predicted the asymptotic frequency with which several irreducible polynomials \emph{simultaneously} achieve prime values; the Hardy--Littlewood Prime Tuples Conjecture is a special case.  The Bateman--Horn Conjecture, even more than Bunyakovsky's Conjecture, seems far out of reach with current methods.

A related family of problems concerns sums of the Liouville function along polynomials.  Recall that the Liouville function $\lambda: \mathbb{N} \to \{-1,1\}$ is the completely multiplicative function that equals $-1$ at all of the primes.  Cancellation in sums of the Liouville function are closely related to the density of primes and the Riemann Hypothesis.  In 1965, Chowla \cite{Ch65} conjectured that 
\begin{equation*}
    \sum_{n\leq X} \lambda(f(n)) = o(X)
\end{equation*}
for any $f(x)\in \Z[x]$ that cannot be written as a constant times the square of another polynomial. The problem in general is still wide open, although there has been some recent progress in the case where $f$ is a product of linear factors. See the further discussion in Section~\ref{sec: sign}. 

Although we cannot prove the Bateman--Horn Conjecture or the Chowla Conjecture for particular polynomials $f$, we are able to describe what happens for a ``random'' $f$.  For our model, we fix $d \in \N$ and consider a uniformly random polynomial $f \in \Z[x]$ of degree at most $d$ with integer coefficients of size at most $H$.  With a sufficiently large amount of averaging in $f$, we can prove an ``almost-all'' version of the Bateman--Horn Conjecture, and we can precisely describe the lower-order fluctuations for the Chowla Conjecture.  Our averaged Bateman--Horn results also let us show that 100\% of polynomials $f$ satisfy a suitable polynomial analogue of the Poisson Tail Conjecture. 
The moral of this paper is that all of the usual ``prime number quasirandomness'' heuristics hold, and to great precision, if one includes enough averaging.  This can be viewed as weak evidence towards the original non-averaged versions of the conjectures.

The main technical idea, which we will expand on shortly, is that while the expression $a_0 + \hdots + a_d x^d$ is a degree-$d$ polynomial in $x$, it is \emph{linear} in $a_0,\hdots, a_d$ for each fixed $x$. In our proofs, we will be able to switch the order of summation so that the quantities of interest are expressed as average correlations of $\Lambda$ or $\lambda$ evaluated at such linear forms (for various $x$'s).  At this point the tools of higher-order Fourier analysis can be fruitfully deployed.  A particularly important input is Leng's work \cite{Leng} on the higher-order Fourier uniformity of the von Mangoldt and M\"obius functions, which combines the techniques of Green and Tao \cite{GT2,GT3,GT4} with recent quantitative improvements on the Green--Tao--Ziegler machinery \cite{green2012inverse} due to Leng~\cite{leng2023efficient} and Leng, Sah, and Swahney~\cite{leng2024quasipolynomial}. This idea also occurs in the work of Ter\"av\"ainen \cite{Teravainen:Chowla}; it is more potent in our work because of recent quantitative improvements that were not available at the time of \cite{Teravainen:Chowla}.

\subsection{Averaged Bateman--Horn}
The earliest averaged Bateman--Horn result is due to Skorobogatov and Sofos \cite{SkorobogatovSofos}, who in 2020 proved that asymptotically 100\% of fixed-degree polynomials satisfy \eqref{eq: Bateman Horn statement}; among the arithmetic applications of this breakthrough is the fact that as a positive proportion of generalized Ch\^atelet surfaces and a positive proportion of a family of conic bundles have rational points. Let us state the result of \cite{SkorobogatovSofos}.  Let $d \in \mathbb{N}$ and let $\delta>0$ be a small constant, and set $H=H(X):= \exp(X^\delta)$.  Let $\mathcal{P}(d,H)$ denote the set of integer-coefficient polynomials of degree at most $d$ where all coefficients have absolute value at most $H$.  Skorobogatov and Sofos showed that
\begin{equation*}
    \E_{f \in \mathcal{P}(d,H)} \left|\frac{1}{X}\sum_{n\leq X} \Lambda(f(n)) - \fS_f'\right| \ll \frac{1}{\sqrt{\log X}},
\end{equation*}
where $\fS_f'$ is a suitable truncated version of the singular series $\fS_f$.  It is desirable to prove such an inequality with $H$ growing more slowly in terms of $X$ (corresponding to less averaging).  Browning, Sofos, and Ter\"av\"ainen \cite{BrowningSofosTeravainen} later achieved this with $H:= X^c$ for $c>\frac{19d}{5}$ with polylogarithmic savings; their work also handles the setting where the coefficients of $f$ are constrained to certain combinatorial cubes. 
See \cite{BalestrieriRome} for further work in this direction.

This brings us to our first result, which provides more precise information in the regime of $H$ exponential in $X$.  We are able to control not only the first moment of the expression of interest but also all higher moments.

\begin{thm}\label{thm: moments of BH}
Let $d \geq 1$ be an integer, and let $0<\delta<1$ be a real. Set $H=H(X): = \exp(X^\delta)$. Then for every $k\in \N$, we have 
\begin{equation}\label{eq:average-BH-statement}
    \E_{f \in \mathcal{P}(d,H)} \left(\frac{1}{X}\sum_{n\leq X} \Lambda(f(n)) - \fS_f\left(\frac{\log X}{\log\log X}\right)\right)^k = O_{d,\delta,k}((\log X)^{-1+o(1)}),
\end{equation}
where the truncated singular series is given by 
\begin{equation*}
    \fS_f\left(\frac{\log X}{\log\log X}\right) := \prod_{p\leq (\log X)/ (\log\log X) } \left(\frac{p^{-1} \#\{x\in \F_p: f(x)\in \F_p^\times\}}{1-1/p}\right). 
\end{equation*}
\end{thm}
Let us pause to highlight a few aspects of this theorem statement; these remarks also apply to our later results.
\begin{enumerate}
    \item Recall that the heuristic \eqref{eq: Bateman Horn statement} pertains only to irreducible polynomials $f$.  It has been known since the work of van der Waerden \cite{van36} that a uniformly random polynomial $f \in \mathcal{P}(d,H)$ is irreducible with probability $1-\exp(-\Omega_d((\log H)/(\log\log H)))$.  Thus reducible polynomials $f$ make a negligible contribution to \eqref{eq:average-BH-statement}.  
    \item The truncated singular series is nonzero as long as there is no prime $p\leq (\log X)/(\log\log X)$ such that $f(x)\equiv 0 \bmod p$ for all $x$. For $p>d$, we have $f\equiv 0 \bmod p$ if and only if $p$ divides all of the coefficients of $f$, so our truncated singular series is nonzero for a positive proportion (depending on $d$) of polynomials $f\in \calP(d,H)$. 
 In particular, \eqref{eq:average-BH-statement} provides a meaningful statement about the $\fS_f \neq 0$ case.
    \item We truncate our singular series because of issues related to its absolute convergence; at the present, such a truncation seems necessary.  The authors of \cite{SkorobogatovSofos,BrowningSofosTeravainen} also truncate their singular series, but at different spots ($\log X$ and $\exp(\sqrt{\log X})$, respectively) from ours. We chose the truncation spot $\log X/\log\log X$ in order to make our computations cleaner.
    \item Our proof would still work if instead of averaging over all of $\mathcal{P}(d,H)$, we fixed all but two of the coefficients of $f$ and then averaged over just the last two coefficients.  The work of \cite{BrowningSofosTeravainen} has the same feature, even though their proof techniques there are quite different.
    \item There is no real difference between averaging over polynomials of degree exactly $d$ and averaging over polynomials of degree at most $d$, since only a $O(1/H)$-fraction of polynomials in $\mathcal{P}(d,H)$ have degree strictly smaller than $d$.
    \item We can obtain the same conclusion with $H(X)$ as large as around $\exp(X/(\log X)^{k+1})$. 
\end{enumerate}

\begin{remark}
 Browning, Sofos and Ter\"av\"ainen \cite{BrowningSofosTeravainen} proved a version of Theorem \ref{thm: moments of BH} in the (more interesting and difficult) polynomial regime where $X \in [H^c, 2H^c]$ for $0<c<19d/5$. This does not immediately imply our result in the exponential regime, due to the upper bound needed on $H$.
\end{remark}

\subsection{Averaged Prime Tuples and Poisson Tail}

Theorem~\ref{thm: moments of BH} pertained to the one-polynomial case of the Bateman--Horn Conjecture.  Our methods also are applicable to many instances with multiple polynomials.  We will focus on what might be termed the ``Hardy--Littlewood Prime Tuples Conjecture along polynomials'' because this case has interesting consequences for gaps between prime values of polynomials. Our averaged version goes as follows.

\begin{thm}\label{thm: BH prime tuple hybrid}
Let $d \geq 1$ be an integer and let $0<\delta<1$ be a real. Set $H=H(X): = \exp(X^\delta)$.  Let $1 \leq k < 1/\delta$ be an integer, and let $\alpha>0$ be a real.  Then for any distinct integers $\ell_1,\ldots,\ell_k$ of size $|\ell_i| \leq X^\alpha$, we have
\begin{equation*}
    \E_{f \in \mathcal{P}(d,H)} \left|\frac{1}{X}\sum_{n\leq X} \prod_{i=1}^k \Lambda(f(n+\ell_i)) - \fS_{f,\Vec{\ell}}\,\left(\frac{\log X}{\log\log X}\right)\right|^2 = O_{d,\delta,\alpha,k}((\log X)^{-1+o(1)}),
\end{equation*}
where the truncated singular series is given by
\begin{equation*}
    \fS_{f,\Vec{\ell}}\,\left(\frac{\log X}{\log\log X}\right) := \prod_{p\leq \log X/(\log\log X)} \left(\frac{p^{-1} \#\{x\in \F_p: f(x+\ell_i)  \in \F_p^\times ~ \forall 1\leq i\leq k\}}{(1-1/p)^k}\right).
\end{equation*}
\end{thm}
\begin{remark}\label{rem: diagonalization}
A standard diagonalization argument produces a function $\delta(X)=\delta_{d,\alpha}(X)$, tending very slowly to $0$ as $X \to \infty$, such that the following holds with $H(X):=\exp(X^{\delta(X)})$:  For any $k \in \mathbb{N}$ and any distinct integers $\ell_1<\cdots<\ell_k$ of size at most $X^\alpha$, we have
    \begin{equation*}
        \E_{f \in \mathcal{P}(d,H)}\left|\frac{1}{X}\sum_{n\leq X} \prod_{i=1}^k \Lambda(f(n+\ell_i)) - \fS_{f,\Vec{\ell}}\,\left(\frac{\log X}{\log\log X}\right)\right|^2 = O_{d,\alpha,k}((\log X)^{-99/100}).
    \end{equation*}
This inequality holds uniformly in the $\ell_i$'s for each fixed $k$ (but of course it is non-uniform in $k$).  This formulation is useful because it sets us up to study gaps between prime values assumed by $f$; we emphasize that the uniformity in the shifts $\ell_i$ will be important for this application.
\end{remark}

Gallagher \cite{Gallagher} and Montgomery and Soundararajan \cite{MontgomerySound} showed that if one assumes a strong form of the Hardy--Littlewood Prime Tuples Conjecture, then one can show that primes are distributed like Poisson random variables in short intervals and like Gaussian random variables in longer intervals.
Given a polynomial $f$, one can ask about the distribution of integer inputs on which $f$ achieves prime values; of course the general case of this problem is even more difficult than the case $f(x)=x$ that we just described.

By combining the approach of Gallagher with the (diagonalized version of) Theorem~\ref{thm: BH prime tuple hybrid}, we are able to establish \emph{unconditional} almost-all results for this polynomial prime gaps problem.  The setup is a bit delicate.  For $f$ a polynomial and $X,L$ natural numbers, define the counting function
\begin{equation*}
        P_f(X,L;k) := \#\{1 \leq x\leq X: f([x,x+L)) \textrm{ contains exactly $k$ primes}\}.
    \end{equation*}
We obtain a probability distribution $\rho_{f,X,L}$ on $\Z_{\geq 0}$ by setting $\rho_{f,X,L}(k):=X^{-1}P_f(X,L;k)$.
Since $f(X) \asymp HX^d$ for typical $f \in \mathcal{P}(d,H)$, the Poisson Tail Conjecture heuristic suggests that $\rho_{f,X,L}$ should be close to the Poisson distribution with mean $$\mathcal{L}:=\frac{\mathfrak{S}_f\left(w(X)\right) L}{\log(HX^d)},$$  where $w(X)$ is a suitable slowly-growing function and $\fS_f(w(X))$ is a truncated singular series as in Theorem~\ref{thm: moments of BH}.  We will show that this is indeed the case asymptotically for 100\% of polynomials $f \in \mathcal{P}(d,H)$. Note that $\fS_f(w)$ is nonzero with a positive probability; of course, on the positive-probability event that $\fS_f(w)$ does vanish, determining $\rho_{f,X,L}$ is trivial.

\begin{thm}\label{thm: Poisson distribution}
Let $d \geq 1$ be an integer, and let $\delta(X)=\delta_{d,1}(X)$ be the function from Remark~\ref{rem: diagonalization}.  Let $\mathcal{L}>0$.
Let $X_1, X_2, \ldots$ be a sequence of natural numbers satisfying $X_i \geq \exp(i^{1+\gamma})$ for some $\gamma>0$. Set $H_i:= \exp(X_i^{\delta(X_i)})$ and $w_i:=\delta(X_i)(\log X_i)/(\log \log X_i)^2$; for each $i \in \mathbb{N}$, sample $f_i$ uniformly at random from the polynomials $f \in \mathcal{P}(d,H_i)$ with $\mathfrak{S}_f(w_i) \neq 0$, and set $L_i:=\frac{\mathcal{L}\log(H_iX_i^d)}{\mathfrak{S}_{f_i}(w_i)}$.  Then with probability $1$, for all $k \in \mathbb{Z}_{\geq 0}$ we have 
    \begin{equation*}
       \lim_{i \to \infty} \frac{P_{f_i}(X_i,L_i;k)}{X_i}=\frac{e^{-\mathcal{L}}\mathcal{L}^k}{k!}.
    \end{equation*}
\end{thm}

\begin{remark}
The conclusion of this conjecture also implies the following weaker statement along the lines of what sometimes goes under the name of the Poisson Tail Conjecture.  For each $i$, let $n_{i,1}<\cdots<n_{i,R_i}$ be the integers $1 \leq n \leq X_i$ such that $f_i(n_{i,j})$ is prime.  Then for any $\lambda>0$, with probability $1$ we have
$$\#\{1 \leq j<R_i: |f(n_{i,j+1})-f(n_{i,j})| \geq \lambda \log (H_i X_i^d)\} \asymp e^{-\lambda} \frac{X_i}{\log (H_i X_i^d)}.$$

\end{remark}

In the regime where $\mathcal{L}$ tends to infinity, one would expect $\rho_{f,X,L}$ to be close to a Gaussian distribution with mean $\mathcal{L}$.  We can, once again, establish that this is the case asymptotically asymptotically 100\% of polynomials $f \in \mathcal{P}(d,H)$, for a reasonably wide range of parameters.

\begin{thm}\label{thm:primes-in-intervals-gaussian}
Let $d \geq 1$ be an integer, let $\alpha>0$, and let $\delta(X)=\delta_{d,\alpha}(X)$ be the function from Remark~\ref{rem: diagonalization}. Let $\mathcal{L}=\mathcal{L}(X)>0$ be a function tending to infinity with $X$ and satisfying $\log \mathcal{L}(X)=o(\log(\delta(X)(\log X)/(\log\log X)^2))$.
Let $X_1, X_2, \ldots$ be a sequence of natural numbers satisfying $X_i \geq \exp(i^{1+\gamma})$ for some $\gamma>0$, and and set $H_i:= \exp(X_i^{\delta(X_i)})$ and $w_i = \delta(X_i)(\log X_i)/(\log\log X_i)^2$; for each $i \in \N$, sample $f_i $ uniformly at random from the polynomials $f \in \mathcal{P}(d,H_i)$ with $\mathfrak{S}_f(w_i) \neq 0$, and set $p_i:=\frac{\mathfrak{S}_{f_i}(w_i)}{\log(H_iX_i^d)}$ and $L_i:=p_i^{-1}\mathcal{L}(X_i)$.  Then with probability $1$, for all $t \in \mathbb{R}$ we have 
    \begin{equation*}
       \lim_{i \to \infty} \frac{\sum_{k \leq p_i L_i+t\sqrt{(p_i-p_i^2)L_i}}P_{f_i}(X_i,L_i;k)}{X_i}=\frac{1}{\sqrt{2\pi}}\int_{-\infty}^t e^{-x^2/2} \,dx.
    \end{equation*}

\end{thm}

\begin{remark}
In the classical setting of primes in short intervals, it is expected that $$\#\{1 \leq x\leq X: [x,x+L]\textrm{ contains exactly $k$ primes}\}=0$$
if $L$ is sufficiently large with respect to $k$ and $X$. 
 Moreover, this quantity should become harder to estimate when $L$ is larger relative to $k$ . Our work corresponds to the regime $L = \log X \cdot (\log\log\log X)^A$ for arbitrary $A>0$ in the classical setting.

If for a suitable choice of $w(X)$ one could replace the upper bound in Theorem~\ref{thm: BH prime tuple hybrid} with square-root cancellation, namely,
\begin{equation*}
    \E_{f\in \calP(d,H)}\left|\frac{1}{X}\sum_{n\leq X} \prod_{i=1}^k \Lambda(f(n+\ell_i)) - \fS_{f,\Vec{\ell}}\,(w(X))\right|^2 = O_{d,\delta,\alpha}(X^{-1/2+o(1)}),
\end{equation*}
then we could extend the range of $\mathcal{L}$ to $\mathcal{L}(X) \ll X^{A\delta(X)}$ for any $A>0$. This would correspond to $L = \log X \cdot (\log \log X)^A$ for $A>0$ in the classical setting. It may be of interest to note that Montgomery and Soundararajan \cite{MontgomerySound} work with a similar uniform square-root cancellation hypothesis in their study of the Gaussian range in the classical setting.  
\end{remark}

\subsection{Averaged polynomial Chowla}

We now turn our attention from the von Mangoldt function $\Lambda$ to the Liouville function $\lambda$.  In the Bateman--Horn-type problems above, we identified main terms (via singular series) in various averages of $\Lambda$.  In Chowla-type problems, the main term is typically zero, and the interesting problems concern the rate of decay and the ``shape'' of lower-order terms.  Instead of establishing ``almost-all'' results as we vary the polynomial $f$, we will view $f$ as a random variable and study the limiting distributions (over $f$) of various Liouville sums involving $f$.

The overarching theme of this subsection is that one should expect $\lambda$ to behave like a uniformly random $\{-1,1\}$-valued sequence, and one can compare its statistics to the corresponding statistics of a truly random $\{-1,1\}$-valued sequence.  For example, the sum of $N$ independent uniformly random $\pm 1$'s is typically of size $\asymp N^{1/2}$, and the limiting distribution of these fluctuations (after normalization) is Gaussian.

The first results in this direction were due to Browning, Sofos, and Ter\"av\"ainen \cite{BrowningSofosTeravainen} and Ter\"av\"ainen \cite{Teravainen:Chowla}; Wilson \cite{Wilson} later studied the higher moments of the expression $$\frac{1}{X^{1/2}}\sum_{n \leq X}\lambda(f(n)),$$ considered as a random variable depending on the random polynomial $f \in \mathcal{P}(d,H)$. Since the Gaussian distribution is determined by its moments, one can deduce that the random variable $X^{-1/2}\sum_{n \leq X}\lambda(f(n))$ converges to a Gaussian (as a distribution) if one knows that all of its individual moments converge to the corresponding moments of a Gaussian.  Wilson made progress towards this goal by using the circle method to show that indeed the first $d+1$ moments of $X^{-1/2}\sum_{n \leq X}\lambda(f(n))$ approach the first $d+1$ moments of a Gaussian of variance $1$.

We mentioned previously that one can view a collection of polynomials $\{a_0+\hdots+a_d n_i^d\}_{i=1}^t$ as a system of linear form in the variables $a_0,\hdots,a_d$.  The overdetermination of this system for $t>d+1$ limits Wilson's circle-method approach to the first $d+1$ moments.  By replacing Wilson's Fourier-analytic approach with higher-order Fourier-analytic tools, we are able to obtain all moments.  Let $C_k:=k!! \cdot {\bf 1}_{\text{$k$ even}}$ denote the $k$-th moment of the Gaussian of mean $0$ and variance $1$.

\begin{thm}\label{thm: moments of polynomial Chowla}
    Let $d \geq 1$ be an integer. Suppose $H=H(X)\geq \exp(X^\delta)$ for some real $\delta>0$.  Then for every $k\in \N$, we have 
\begin{equation*}
    \E_{f \in \mathcal{P}(d,H)} \left(\frac{1}{X^{1/2}}\sum_{n\leq X}\lambda(f(n))\right)^k = C_k + O_{d,\delta,k}(X^{-1}).
\end{equation*}
\end{thm}

An immediate corollary is that $X^{-1/2}\sum_{n \leq X}\lambda(f(n))$ converges in distribution to a Gaussian of mean $0$ and variance $1$, i.e.,
$$\lim_{X \to \infty} \Pr_{f \in \mathcal{P}(d,H)}\left( \frac{1}{X^{1/2}}\sum_{n \leq X}\lambda(f(n)) \leq t \right)=\frac{1}{\sqrt{2\pi}}\int_{-\infty}^t e^{-x^2/2} \,dx$$
for all $t \in \mathbb{R}$.

Theorem~\ref{thm: moments of polynomial Chowla} is a special case of a more general result about sign patterns of the Liouville function sampled along a random polynomial.  Recall that the sign-pattern instance of the Chowla Conjecture asserts that the sequence $\lambda(n+1), \ldots, \lambda(n+s)$ assumes each of the $2^s$ possible sign patterns for $(1+o(1))2^{-s}X$ values of $1 \leq n \leq X$.  For a polynomial $f$ that is not a constant times a square of another polynomial, one can make the analogous conjecture about the sign patterns of the sequence
$$\lambda(f(n+1)), \lambda(f(n+2)), \ldots, \lambda(f(n+s)).$$
As usual, we are able to address this conjecture only on average; as in Theorem~\ref{thm: moments of polynomial Chowla}, we show not only that the main term is as expect but also that the lower-order fluctuations of these counts converge to Gaussians.

\begin{thm}\label{thm: sign patterns}
 Let $d \geq 1$ be an integer. Suppose $H=H(X)\geq \exp(X^\delta)$ for some $\delta>0$.  Let $\Vec{\epsilon}=(\epsilon_1,\hdots,\epsilon_s) \in \{-1,1\}^s$ be a sign pattern. Then for every $k\in \N$, we have 
    \begin{equation*}
        \E_{f \in \mathcal{P}(d,H)}\left(\frac{\#\{n\leq X: (\lambda(f(n+1)),...,\lambda(f(n+s))) = \Vec{\epsilon}~\}-2^{-s}X}{X^{1/2}}\right)^k  = \sigma(\Vec{\epsilon})^k C_k + O_{d,\delta,s,k}(X^{-1}),
    \end{equation*}
    where the standard deviation $\sigma(\Vec{\epsilon}) \geq 0$ is given by
$$\sigma(\Vec{\epsilon})^2=4^{-s} \left( \sum_{\substack{\emptyset \neq T_1,T_2 \subseteq [s]\\ \text{translates}}} \prod_{i \in T_1} \epsilon_i \cdot \prod_{i \in T_2} \epsilon_i  \right).$$ 
\end{thm}

Of course, it follows that the main quantity under consideration is distributed as a Gaussian of mean $0$ and variance $\sigma(\Vec{\epsilon})^2$ in the limit.  The $s=1$ case corresponds to Theorem~\ref{thm: moments of polynomial Chowla} (via the identity $\sum_{n \leq X}\lambda(f(n))=X-2\#\{n \leq X: \lambda(f(n))=-1\}$).

In the course of the proof of Theorem~\ref{thm: sign patterns}, we will see that if $y(1), y(2), \ldots$ is a sequence of independent uniformly random $\pm 1$'s, then the limiting distribution of the count
$$\frac{\#\{n\leq X: (y(n+1), \ldots, y(n+s)) = (\epsilon_1,...,\epsilon_s)\}-2^{-s}X}{X^{1/2}}$$
is also a Gaussian of mean $0$ and variance $\sigma(\overline{\varepsilon})^2$, in accordance with the heuristic from the beginning of this subsection.  The latter fact follows from a more general classical result of Hoeffding and Robbins \cite{hoeffding1994central} on $m$-dependent random variables.

\subsection{Related questions}\label{subsec: related questions}
There are several ways in which one could hope to extend our results.  Our argument for Theorem \ref{thm: BH prime tuple hybrid} should also give $L^q$-bounds with $q>2$, possibly under further restrictions among $\delta,k,q$; we worked out only the $L^2$-bounds because these suffices for our Poisson Tail applications. The hypotheses $\delta<1$ in Theorem \ref{thm: moments of BH} and $\delta<1/k$ in Theorem \ref{thm: BH prime tuple hybrid} seem to reflect a technical limitation of our methods, and it would be interesting to remove or weaken these assumptions.

A natural follow-up question is whether Theorems~\ref{thm: moments of BH} and~\ref{thm: BH prime tuple hybrid} can be established with $H(X)$ growing more slowly, e.g., polynomially in $X$. 
\cite[Theorem 1.2]{BrowningSofosTeravainen} directly implies a version of Theorem \ref{thm: moments of BH} where $H(X)$ grows polynomially in $X$. Indeed, their result shows that for $0<c<5/(19d)$ and for $H \asymp X^{1/c}$, for any $k\in \mathbb{N}$ and $A>0$, we have
\begin{equation*}
    \E_{f\in \mathcal{P}(d,H)} \left(\frac{1}{X} \sum_{n\leq X} \Lambda(f(n)) - \fS_f(\exp(\sqrt{\log X}))\right)^k  = O_{d,c,k,A}((\log X)^{-A}). 
\end{equation*}
The techniques of Browning, Sofos and Ter\"av\"ainen differ greatly from ours; one topic for future work could be seeing what their methods say about analogues of the Hardy--Littlewood Prime Tuples Conjecture and the Poisson Tail Conjecture. Another possible path towards decreasing $H(X)$ is utilizing the ``Siegel zero model'' for the von Mangoldt function (described explicitly in \cite{MTW:Siegel_correction}) and exploiting the fact that we are averaging over many linear systems (indexed by the variables $n_i$ below).

\subsection{Structure of the paper}
In Section~\ref{sec:prelim} we gather notation, our singular series conventions, and other preliminaries and background.  In Section~\ref{sec:hofa} we explain the necessary inputs from higher-order Fourier analysis.  Section~\ref{sec: poly chowla} contains our on-average result about the polynomial Chowla problem and Liouville sign patterns.  We upper-bound all moments for Bateman--Horn in Section~\ref{sec:BH-moments}.  We prove our almost-all prime tuples result in Section~\ref{sec:tuples}, and in Section~\ref{sec:tail} we establish our almost-all results on Poisson and Gaussian tails.

\section{Preliminaries and heuristics}\label{sec:prelim}

\subsection{Notation}\label{subsec: notation}
Throughout this paper, $\Lambda, \lambda, \mu$ denote the von Mangoldt, Liouville, and M\"obius functions, respectively.  We extend each of these functions to all of $\mathbb{Z}$ by declaring it to be equal on $n,-n$ for all $n$.  We write $\mathscr{P}$ for the set of primes (again both positive and negative).

We will always treat the degree $d$ as fixed, and we will often treat the moment or tuple length $k$ as fixed as well.  In proofs, we will sometimes omit the dependence on $d,k$ in our asymptotic notation $O(\cdot)$, $o(\cdot)$, $\ll$ when there is no risk of confusion. 

\subsection{Singular series}

In general, we will write $\fS$ for a singular series and $\fS(w)$ for a truncation of it. Singular series depending on polynomials $f(x)\in \Z[x]$ often need to be truncated due to issues about absolute convergence.

Let us write down some singular series that we will encounter later. First,the Bateman--Horn heuristic involves the singular series
\begin{equation*}
    \fS_f:= \prod_{p} \left(\frac{p^{-1}\#\{x\in \F_p: f(x)\in \F_p^\times \}}{1-1/p}\right).
\end{equation*}
Second, the singular series from the Hardy--Littlewood Prime Tuples Conjecture is
\begin{equation*}
    \fS_{\Vec{\ell}}:= \prod_{p} \left(\frac{p^{-1}\#\{x\in \F_p: x+\ell_i\in \F_p^\times~ \forall i\}}{(1-1/p)^{|\Vec{\ell}|}}\right).
\end{equation*}
Third, our ``hybrid'' problem about prime tuples along polynomials corresponds to the singular series
\begin{equation*}
    \fS_{f,\Vec{\ell}} = \prod_{p} \left(\frac{p^{-1}\#\{x\in \F_p: f(x+\ell_i)\in \F_p^\times~ \forall i\}}{(1-1/p)^{|\Vec{\ell}|}}\right).
\end{equation*}
We will work with truncated versions of the first and third of these singular series, viz.
\begin{equation*}
    \fS_f(w):= \prod_{p \leq w} \left(\frac{p^{-1}\#\{x\in \F_p: f(x)\in \F_p^\times \}}{1-1/p}\right),
\end{equation*}

\begin{equation*}
    \fS_{f,\Vec{\ell}}\,(w):= \prod_{p \leq w} \left(\frac{p^{-1}\#\{x\in \F_p: f(x+\ell_i)\in \F_p^\times~ \forall i\}}{(1-1/p)^{|\Vec{\ell}|}}\right),
\end{equation*}
where the truncation point $w$ is a slowly-growing parameter.  We need the following basic bounds.

\begin{lemma}\label{lem: upper bound on truncated singular series}
Let $f(x)\in \Z[x]$, and let $\Vec{\ell}\in \Z^k$ be a tuple of distinct integers. Then
    \begin{equation*}
        \fS_f(w) \ll \log w \quad \text{and} \quad \fS_{f,\Vec{\ell}}\,(w)\ll (\log w)^k.
    \end{equation*}
\end{lemma}
\begin{proof}
    From $\#\{x\in \F_p: f(x)\in \F_p^\times\} ,\,\#\{x\in \F_p: f(x+\ell_i)\in \F_p^\times~ \forall i\} \leq p$ we deduce that 
    \begin{equation*}
        \fS_f(w) \leq \prod_{p\leq w} \left(1-\frac{1}{p}\right)^{-1} \ll \exp\left(\sum_{p\leq w} \frac{1}{p}\right)\ll \log w
    \end{equation*}
   and likewise
    \begin{equation*}
        \fS_{f,\Vec{\ell}}\,(w) \leq \prod_{p\leq w} \left(1-\frac{1}{p}\right)^{-k} \ll \exp\left(k \sum_{p\leq w} \frac{1}{p} \right) \ll (\log w)^{k}.\qedhere
    \end{equation*}
\end{proof}

\begin{lemma}\label{lem: lower bound truncated singular series}
    Let $f(x)\in \Z[x]$ be a polynomial of degree $d$. If there is a prime $p\leq w$ dividing $f(n)$ for all integer inputs, then $\fS_f(w) = 0$. Otherwise, we have the lower bound 
    \begin{equation*}\fS_f(w) \gg_d (\log w)^{1-d}.\end{equation*}
\end{lemma}
\begin{proof}
    For each prime $p$, if $f$ modulo $p$ is not the zero polynomial, then it vanishes on at most $d$ elements of $\F_p$.  In particular, 
    $\#\{x\in \F_p: f(x)\in \F_p^\times\}$ is either $0$ or at least $p-d$.  If the former occurs for some $p \leq w$, then $\fS_f(w)=0$.  Otherwise we have
    \begin{equation*}
        \fS_f(w) \geq \prod_{p\leq w} \left(\frac{1-\min(d,p-1)/p}{1-1/p}\right) \gg_d \prod_{p\leq w}\left(1-\frac{1}{p}\right)^{d-1} \gg_{d} (\log w)^{1-d}.\qedhere
    \end{equation*}
\end{proof}

Finally, we record the singular series that appears in the main term of the asymptotic results of Green and Tao \cite{GT2,GT4,GT3}, and Green, Tao, and Ziegler \cite{green2012inverse} on linear equations in primes. Let $\Psi= (\psi_1,\hdots,\psi_k)$ denote a set of linear forms in $n$ variables with finite complexity. The corresponding singular series is
\begin{equation*}
    \fS_{\Psi}:= \prod_{p} \left(\frac{p^{-n} \#\{\Vec{x}\in \F_p^n: \psi_i(\Vec{x})\in \F_p^\times~ \forall i\}}{(1-1/p)^k}\right).
\end{equation*}
Observe that $\fS_{\Psi}$ converges absolutely. Indeed, for all but finitely many primes $p$, we have $$\#\{\vec x\in \F_p^n:\psi_i(\vec x)\in \F_p^\times~\forall i\} = p^n -kp^{n-1}+O_{\Psi}(p^{n-2})$$ from the Inclusion-Exclusion Principle and the fact that each equation $\psi_i(\vec x)=0$ cuts out a plane.

\subsection{Heuristics for sign patterns of the Liouville function}\label{sec: sign}
It is believed that the sequence of Liouville values $\lambda(n)$ should behave ``similarly'' to a random $\pm 1$-valued sequence.  A concrete conjecture in this direction is that a length-$s$ subinterval of Liouville values with a random starting point is equally likely to see each of the $2^s$ possible sign patterns.

\begin{conjecture}[Chowla]\label{conj:chowla}
    Let $s\geq1$ and let $(\epsilon_1,\hdots,\epsilon_s)\in \{-1,1\}^s$ be a sign pattern.  Then the set $\{n \in \mathbb{N}: \lambda(n+i) = \epsilon_i ~\forall i\}$ has natural density $2^{-s}.$
\end{conjecture}

The $s=1$ case is equivalent to the prime number theorem. Hildebrand \cite{Hil86} showed that each sign pattern of length at most $3$ occurs infinitely often, and in 2015 Matom\"aki, Radziwi\l{}\l{}, and Tao \cite{MRT16} established that these sign patterns occur with positive natural lower density. The same authors \cite{MRT:averaged_Chowla} studied the distribution of the tuple $(\lambda(n+h_1),\ldots, \lambda(n+h_s) )$ when the $h_i$'s are averaged over an interval. 

In a related direction, there has been much recent progress on the logarithmically-weighted form of Chowla's Conjecture, namely, the statement that
\begin{equation*}
    \frac{1}{\log X} \sum_{n\leq X} \frac{\lambda(n+1)\cdots \lambda(n+s)}{n} = o_s(1)
\end{equation*}
as $X\rightarrow\infty$.
This conjecture is strictly weaker than Conjecture \ref{conj:chowla}. Tao \cite{Tao:log_chowla} proved the logarithmically-weighted Chowla Conjecture for $s=2$, and Tao and Ter\"av\"ainen \cite{TT:log_Chowla} later proved it for all odd $s$. The same authors showed that for $s=4$ all sign patterns occur with positive lower density. See also \cite{helfgott2021expansion,Pilatte}. Recent work towards the higher-uniformity conjecture \cite{MRTTZ23,  MSTT:higher_uniformity,MRSTT:higher_uniformity,Walsh:phase_pyramids,Walsh:local_phases} is approaching the range that would establish the logarithmically-weighted Chowla conjecture for all $s$. Sawin \cite{Sawin:sign_patterns} has explained potential obstructions to obtaining more detailed information on sign patterns using this line of thought. 

The full polynomial Chowla Conjecture generalizes the above sign pattern conjecture.
\begin{conjecture}[Chowla]
Let $f(x)\in \Z[x]$.  If $f$ is not a constant times the square of another polynomial, then as $X\rightarrow\infty$ we have 
$$\frac{1}{X}\sum_{n\leq X} \lambda(f(n)) = o_f(1).$$
\end{conjecture}

Ter\"av\"ainen \cite{Teravainen:Chowla} showed that if $f$ factors as a product of linear and quadratic polynomials and is not a constant times a perfect square, then $\lambda(f(n))$ attains each of $-1,1$ for a positive proportion of inputs. He also established  an averaged version of the polynomial Chowla problem; this was improved upon quantitatively by \cite{BrowningSofosTeravainen}.  These results hold more generally for nonpretentious multiplicative functions and can be understood as multiplicative analogues of the work of Skorobogatov and Sofos \cite{SkorobogatovSofos}.  In the function-field setting, this question has been resolved by the work of Sawin and Shusterman \cite{SawinShusterman:polynomial_chowla}. Another recent development \cite{KSX23} is that for a fixed polynomial $f$, there is square-root cancellation when $\lambda$ is replaced by a random multiplicative function.

\subsection{Distribution of gaps between primes}
Most of this section is based on the survey of Funkhouser, Goldston and Ledoan \cite{FunkhouserGoldstonLedoan} on the distribution of gaps between consecutive primes; any mistakes are, of course, our own. 

Based on Cram\'er's model, one expects the average gap between two consecutive prime numbers $p_n$, $p_{n+1}$ to be of size $\sim \log p_n$. Works of Zhang \cite{Zhang__gaps}, Maynard \cite{Maynard__Gaps}, and independently Tao (in unpublished work), and the Polymath project \cite{polymath__gaps} exhibit explicit constants $c$ such that there are infinitely many prime gaps of size at most $c$. It is less clear what to expect regarding the largest gaps between consecutive primes. Building on the work of \cite{FGKT:large_gaps,Maynard__Gaps}, the authors of \cite{FGKMT:large_gaps} showed that there are gaps of size at least $$\gg \frac{(\log p_n) (\log\log p_n)(\log\log\log\log p_n)}{\log\log\log p_n}.$$ 
The folklore Poisson Tail Conjecture predicts that the distribution of prime gaps should be roughly Poisson at scale $\log p_n$.

\begin{conjecture}\label{conj: Poisson Tail}
    Let $\varepsilon>0$.  For $1\leq H\leq (\log X)^{2-\varepsilon}$, we have
    \begin{equation*}
       \sum_{\substack{p_{n+1}\leq X \\ p_{n+1}-p_n\geq H}} 1 \asymp e^{-H/\log X} \cdot \frac{X}{\log X};
    \end{equation*}
    \begin{equation*}
        \sum_{\substack{p_{n+1}\leq X \\ p_{n+1}-p_n\geq H}} (p_{n+1}-p_n) \asymp \left(1+\frac{H}{\log X}\right) \cdot e^{-H/\log X} \cdot X.
    \end{equation*}
    For $H>(\log X)^{2+\varepsilon}$ and $X$ sufficiently large, both of the above quantities vanish.
\end{conjecture}

Gallagher \cite{Gallagher} studied the regime $H(X)=\lambda \log X$, with $\lambda>0$ constant, under the assumption of a uniform version of the Hardy--Littlewood Prime Tuples Conjecture. 

\begin{thm}[Gallagher] \label{thm: Gallagher}
Let $\lambda>0$, and set $H(X):= \lambda\log X$. 
    Assume that for each fixed integer $k\geq 2$ and each admissible tuple $\Vec{\ell}\in \Z^k$, we have 
    \begin{equation*}
        \#\{1 \leq n\leq X: n+\ell_i \in \mathscr{P} ~\forall i\} = (1+o_k(1))\fS_{\Vec{\ell}}\, \cdot \frac{X}{(\log X)^k}
    \end{equation*}
    uniformly for $\ell_1, \ldots, \ell_k$ of size at most $H$. Then the quantity
    \begin{equation*}
        P(X,H;k) = \#\{1 \leq n\leq X: (n,n+H] \textrm{ contains exactly $k$ primes}\}
    \end{equation*}
   satisfies the asymptotic (as $X \to \infty$)
    \begin{equation*}
        P(X,H;k) \sim \frac{e^{-\lambda}\lambda^k}{k!}\cdot X.
    \end{equation*}
\end{thm}
Montgomery and Soundararajan \cite{MontgomerySound} showed that under the assumption of a strong form of the Hardy--Littlewood Prime Tuples Conjecture ---more precisely, one with both a uniformity assumption and square-root cancellation---the $2k$-th moments of the gaps between primes are consistent with the prediction of the Poisson Tail Conjecture in the range $\log X \leq H\leq X^{1/2k}$. They found that in this range of $H$, the lower-order fluctuations are slightly smaller than what the Cr\'amer model would predict.  There is no such ``unexpected smaller variance'' phenomenon in our work, essentially because our averaging parameter $H$ is so large that the averaging ``covers up'' more subtle lower-order behavior. 

Finally, we mention the thematically related work of Balog \cite{balog1990prime} and Kawada \cite{kawada1993prime} (see also the references therein) on an averaged version of the Prime Tuples Conjecture where one averages over the shifts.

\section{Higher-order Fourier analysis input}\label{sec:hofa}

In this section we recall the necessary tools from higher-order Fourier analysis.  This area originated in Gowers's new proof of Szemer\'edi's Theorem \cite{gowers1998new,gowers2001new}.  The work of Green and Tao \cite{GT2,GT4,GT3}, and later Green, Tao, and Ziegler \cite{green2012inverse}, used higher-order Fourier analysis to study the asymptotic frequency with which linear forms simultaneously assume prime values.  In general, the circle method is useful for problems of this type when the number of variables exceeds the number of linear forms.  Higher-order Fourier analysis allows one to access the regime where the number of linear forms exceeds the number of variables; it is this shift in perspective that lets us overcome the limitations of Wilson's circle-method approach \cite{Wilson}.

Much of the early work in higher-order Fourier analysis came with poor quantitative dependences, and a major theme has been obtaining ``reasonable'' quantitative bounds.  We will use the strongest known quantitative bounds for arithmetic functions evaluated along linear forms.  The main tool is Leng's recent breakthrough \cite{Leng} on the higher-order Fourier-uniformity of the von Mangoldt and M\"obius functions; as we mentioned above, this work builds on quantitative improvements for the inverse theory of the Gowers norms, as developed by Leng~\cite{leng2023efficient} and Leng, Sah, and Sawhney~\cite{leng2024quasipolynomial}.

\subsection{Leng's quantitative result on linear equations in primes}
We cannot quite use Leng's main result out-of-the-box because it is stated for multilinear systems with constant coefficients, whereas we work with multilinear systems with slowly-growing coefficients.  We will instead use the following variant.

\begin{thm}\label{thm:james-strengthened}
Let $m,t,A \in \mathbb{N}$, and let $N \in \N$ be a large parameter.
Let $\Omega \subseteq [-N,N]^m$ be a convex set, and let $\psi_1, \ldots, \psi_t: \mathbb{Z}^m \to \mathbb{Z}$ be affine linear forms of the form
$$\psi_i(\vec n)=\vec n \cdot \dot \psi_i+\psi_i(0),$$
where $\dot \psi_1, \ldots, \dot \psi_m \in \Z^m$ are pairwise linearly independent with coefficients of size at most $(\log N)^A$, and each $|\psi_i(0)| \leq (\log N)^A$.  Then the following holds.

First, we have the asymptotic
$$\sum_{\vec n \in \Omega \cap \Z^m}\prod_{i=1}^t \Lambda(\psi_i(\vec n))=\mathfrak{S}_\Psi\cdot \vol(\Omega)+O_{m,t,A}\left(N^m (\log N)^{-A}\right),$$
where the singular series $\mathfrak{S}_\Psi$ is given by
\begin{equation*}
\mathfrak{S}_\Psi:=
\prod_p \frac{p^{-m}\#\{\vec n \in \F_p^m: \psi_i(\vec n) \in \F_p^\times ~\forall i\}}{(1-1/p)^t}.
\end{equation*}

Second, we have the cancellation
$$\sum_{\vec n \in \Omega \cap \Z^d}\prod_{i=1}^t \mu(\psi_i(\vec n)), \quad \sum_{\vec n \in \Omega \cap \Z^d}\prod_{i=1}^t \lambda(\psi_i(\vec n))\ll_{m,t,A} N^m (\log N)^{-A}.$$
\end{thm}

There are two differences between this theorem and \cite[Theorem 5]{Leng}:
\begin{itemize}
    \item Leng assumes that the $\dot \psi_i$'s have coefficients of constant size.
    \item Leng states his results only for $\Lambda$ and $\mu$ (not for $\lambda$).
\end{itemize}
It is well-known to experts that one can obtain the statement of Theorem~\ref{thm:james-strengthened} from the arguments in \cite{Leng}.  We will briefly indicate the necessary modifications of the proof.

The main work in \cite{Leng} is establishing the Gowers-norm control
$$\|\Lambda-\Lambda_{\cramer, z}\|_{U^{s+1}[N]} \ll_{s,B} (\log N)^{-B}, \quad \|\mu\|_{U^{s+1}[N]} \ll_{s,B} (\log N)^{-B}$$
for all $s \in \N$ and $B>0$.  Here $\Lambda_{\cramer, z}$ denotes the \emph{Cram\'er approximant} of the von Mangoldt function (where $z=z(N):=\exp((\log N)^{1/10})$ is an auxiliary parameter); its exact form is not relevant here.  Leng's deduction of these bounds goes by way of certain \emph{Siegel zero approximants} $\Lambda_{\siegel}, \mu_{\siegel}$ of $\Lambda, \mu$, whose particular forms are (again) not relevant for us.  Leng shows in his Theorem 7 that for each $s \in \N$ there is some $c_s>0$ such that $$\|\Lambda-\Lambda_{\siegel}\|_{U^{s+1}[N]},~\|\mu-\mu_{\siegel}\|_{U^{s+1}[N]}\ll_{s} \exp(-(\log N)^{c_s}).$$
He then applies\footnote{Leng inserts Siegel's Theorem into his application of \cite[Theorem 2.5]{TaoTeravainen}, so the version that he uses in his paper looks slightly different from the version stated in \cite{TaoTeravainen}.  This use of Siegel's Theorem makes the implied constants ineffective.}  a result of Tao and Terav\"av\"ainen \cite[Theorem 2.5]{TaoTeravainen} which says that $$\|\Lambda_{\siegel}-\Lambda_{\cramer,z}\|_{U^{s+1}[N]},~\|\mu_{\siegel}\|_{U^{s+1}[N]} \ll_{s,B} (\log N)^{-B}.$$
The desired Gowers-norm control then follows from the Triangle Inequality for Gowers norms.

Unsurprisingly, we also need the analogous statement $$\|\lambda\|_{U^{s+1}[N]} \ll_{s,B} (\log N)^{-B}$$
for the Liouville function.  This can be deduced from the Gowers-norm control for the M\"obius function as follows.  As usual, choose a prime $5N \leq M \leq 10N$, and work in $\Z/M\Z$ instead of $\Z$; for ease of notation, write $\lambda':=\lambda \cdot {\bf 1}_{[N]}$ and $\mu':=\mu \cdot {\bf 1}_{[N]}$.  Then $$\|\lambda\|_{U^{s+1}[N]} \asymp_s \|\lambda'\|_{U^{s+1}[N]} \asymp_s \|\lambda'\|_{U^{s+1}(\Z/M\Z)},$$ and likewise for $\mu$ (where we evaluate the functions on $\Z/M\Z$ by first lifting to $[M] \subseteq \Z$).
Using the formula $\lambda(n)=\sum_{r: r^2|n} \mu(n/r^2)$ (which applies equally well to the truncations $\lambda', \mu'$), we expand
\begin{align*}
\|\lambda'\|_{U^{s+1}(\Z/M\Z)}^{2^{s+1}} &=\E_{n, h_1, \ldots, h_{s+1} \in \Z/M\Z} \prod_{\epsilon \in \{0,1\}^{s+1}} \lambda'(n+\epsilon \cdot h)\\
 &=\E_n \lambda'(n) \E_{h_1, \ldots, h_{s+1}} \prod_{0 \neq \epsilon \in \{0,1\}^{s+1}} \lambda'(n+\epsilon \cdot h)\\
 &=\E_n \sum_{r: r^2|n} \mu'(n/r^2) \E_{h_1, \ldots, h_{s+1}} \prod_{0 \neq \epsilon \in \{0,1\}^{s+1}} \lambda'(n+\epsilon \cdot h)\\
 &=M^{-1}\sum_{r=1}^{\sqrt{M}} \sum_{m=1}^{M/r^2} \mu'(m) \E_{h_1, \ldots, h_{s+1}} \prod_{0 \neq \epsilon \in \{0,1\}^{s+1}} \lambda'(mr^2+\epsilon \cdot h).
\end{align*}
Let $C>0$ be a constant to be chosen later.  By the Triangle Inequality, the total contribution of $r>(\log M)^C$ is at most
$$M^{-1} \sum_{r=(\log M)^C}^\infty \frac{M}{r^2} \ll (\log M)^{-C}\ll (\log N)^{-C}.$$
Now, the contribution of a single value $r \leq (\log M)^C$ is
\begin{multline*}
M^{-1}\sum_{m=1}^{M/r^2} \mu'(m) \E_{h_1, \ldots, h_{s+1}} \prod_{0 \neq \epsilon \in \{0,1\}^{s+1}} \lambda'(mr^2+\epsilon \cdot h)\\
=\E_{m \in \Z/M\Z}\mu'(m) \cdot {\bf 1}_{[M/r^2]}(m) \E_{h_1, \ldots, h_{s+1}} \prod_{0 \neq \epsilon \in \{0,1\}^{s+1}} \lambda'(r^2(m+\epsilon \cdot h));
\end{multline*}
in passing to the second line we replaced the variables $h_i$ with the variables $h'_i:=r^{-2} h_i$, which still range over $\Z/M\Z$.  The Gowers--Cauchy--Schwarz Inequality (see, e.g., \cite[Exercise 1.3.19]{tao2012higher}), applied to the function $\mu' \cdot {\bf 1}_{[M/r^2]}$ and the dilated functions $\lambda'(r^2 \times \cdot)$, gives that the last centered quantity is of size at most $$\|\mu' \cdot {\bf 1}_{[M/r^2]}\|_{U^{s+1} (\Z/M\Z)} \cdot \|\lambda'(r^2 \times \cdot)\|_{U^{s+1} (\Z/M\Z)}^{2^{s+1}-1} \leq \|\mu' \cdot {\bf 1}_{[M/r^2]}\|_{U^{s+1} (\Z/M\Z)}.$$
Recalling that $\log(M/r)\asymp_C \log N$ and using the Gowers-norm control on $\mu$, we find that
$$\|\mu' \cdot {\bf 1}_{[M/r^2]}\|_{U^{s+1} (\Z/M\Z)} \asymp_s r^{-(s+1)} \|\mu\|_{U^{s+1}[M/r^2]} \ll_{s,C,\kappa} r^{-(s+1)} (\log N)^{-\kappa}$$
for all $\kappa>0$.  The sum over $r \leq (\log M)^B$ is easily $\ll_{s,C,\kappa} (\log N)^{-\kappa}$.  Taking $C=\kappa:=2^{s+1}B$, we conclude that
$\|\lambda\|_{U^{s+1}[N]} \asymp_s \|\lambda'\|_{U^{s+1}(\Z/M\Z)} \ll_{s,B} (\log N)^{-B}$, as desired.

The aforementioned Gowers uniformity bounds are stated for $\Lambda,\Lambda_{\cramer,z}, \mu$ defined to vanish on negative inputs.  The Triangle Inequality for the Gowers norms immediately gives bounds of the same quality for the extensions of these functions to negative integers via $\Lambda(-n):=\Lambda(n)$, etc.

We now have the setup for the deduction of Theorem~\ref{thm:james-strengthened}.  Consider the multilinear operator
$$F(f_1, \ldots, f_t):=\sum_{\vec n \in \Omega \cap \Z^m} \prod_{i=1}^t f_i(\psi_i(\vec n)).$$
Set $Y=Y(N):=(m+1)(\log N)^{A+1}N$, so that $\psi_i([-N,N]^m) \subseteq [-Y,Y]$ for all $i$.  Standard applications of the Cauchy--Schwarz Inequality and monotonicity properties of box norms (see, e.g., \cite[Exercise 1.3.23]{tao2012higher} for the Cauchy--Schwarz applications, \cite[Appendix C]{GT2} for handling the convex body, and \cite[Lemma 3.5(iv)]{kravitz2024quantitative} for the polynomial dependence on the coefficient sizes) show that for any functions $f_i: [-Y,Y] \to \C$, we have
$$F(f_1, \ldots, f_t) \ll_t N^m (\log N)^{O_t(A)} \min_{1 \leq i \leq t} \left(\|f_i\|_{U^{t-1}[\pm Y]} \cdot \prod_{j \neq i}\|f_j\|_{L^\infty} \right).$$
Inserting the Gowers uniformity bounds on $\mu,\lambda$ from above, we deduce that
$$F(\mu, \ldots, \mu),~ F(\lambda, \ldots, \lambda) \ll_{t,B} N^m(\log N)^{O_t(A)} \cdot (\log N)^{-B,}$$
which is $\ll_{t,A}N^m(\log N)^{-A}$ if $B$ is sufficiently large depending on $t,A$. This completes the proof of the second statement of Theorem~\ref{thm:james-strengthened}.

It remains to prove the von Mangoldt part of Theorem~\ref{thm:james-strengthened}.   By expanding $\Lambda=\Lambda_{\cramer,z}+(\Lambda-\Lambda_{\cramer,z})$, we can write $F(\Lambda, \ldots, \Lambda)$ as $F(\Lambda_{\cramer,z}, \ldots, \Lambda_{\cramer,z})$ plus a sum of $O_t(1)$ term of the form $F(g_1, \ldots, g_t)$, where each $g_i \in \{\Lambda, \Lambda_{\cramer,z}, \Lambda-\Lambda_{\cramer,z}\}$ and there is at least one index $i$ with $g_i=\Lambda-\Lambda_{\cramer,z}$.  The functions $\Lambda, \Lambda_{\cramer,z}$ are pointwise $\ll\log Y$ (the latter due to \cite[Lemma 2.4]{TaoTeravainen}).  Thus the considerations of the previous paragraph, the Triangle Inequality, and our Gowers uniformity bound on $\Lambda-\Lambda_{\cramer,z}$ (with $B$ suitably large in terms of $t,A$) together give
\begin{align*}
F(\Lambda, \ldots, \Lambda) &=F(\Lambda_{\cramer,z}, \ldots, \Lambda_{\cramer,z})+O_{t}(N^m (\log N)^{O_t(A)}(\log Y)^{t-1}\|\Lambda-\Lambda_{\cramer,z}\|_{U^{t-1}[\pm Y]})\\
 &=F(\Lambda_{\cramer,z}, \ldots, \Lambda_{\cramer,z})+O_{t,A}(N^m(\log N)^{-A}).
\end{align*}
Tao and Ter\"av\"ainen \cite{TaoTeravainen} have used sieve methods to estimate this main term to good precision.  In particular, their Theorem 5.2 (the coefficients of our linear forms are comfortably small enough) gives that
$$F(\Lambda_{\cramer,z}, \ldots, \Lambda_{\cramer,z})=\fS_{\Psi}\cdot \vol(\Omega)+O_{m,t}(N^m \exp(-(\log N)^{4/5})).$$
Combining the last two equations completes the proof of Theorem~\ref{thm:james-strengthened}.

\subsection{Application to random polynomials}
Theorem~\ref{thm:james-strengthened} is related to average correlations of random polynomials because for $n$ a fixed natural number and $f(x)=a_dx^d+\cdots+a_0$ a random degree-$d$ polynomial, we can interpret $f(n)$ as a linear form in the random variables $a_0, \ldots, a_d$ whose coefficients are powers of $n$.  The following theorem is the technical tool underpinning the arguments in the rest of the paper.  The division into residue classes for the von Mangoldt statement will be useful for later computations involving truncated singular series.

\begin{thm}\label{thm:GT-input}
Let $t \in \N$ and $d \geq 1$, and let $\delta,A>0$.  Let $X>0$ be a large parameter, and set $H:=\exp(X^\delta)$.  Let $n_1, \ldots, n_t$ be distinct positive integers of size $|n_i| \leq X$, and let $M \leq (\log H)^A$ be a natural number.  Then the following holds.

First, for any polynomial $f_0 \in \Z[x]$ of degree at most $d$, we have the asymptotic
\begin{equation*}
        \E_{\substack{f \in \mathcal{P}(d,H)\\ f\equiv f_0 \bmod M}}\prod_{i=1}^t \Lambda(f(n_i)) =\mathfrak{S}_{\Vec{n},f_0 \bmod M} +O_{t,d,\delta,A}((\log H)^{-A}),
     \end{equation*}
where the singular series $\fS_{\Vec{n},f_0 \bmod M}$ is defined by
\begin{equation*}
    \fS_{\Vec{n},f_0 \bmod M}:= \left(\prod_{p\mid M} \prod_{i=1}^t \frac{\mathbf{1}_{f_0(n_i)\neq 0 \bmod p}}{1-1/p}\right) \cdot \left(\prod_{p\nmid M} \frac{p^{-(d+1)}\#\{\Vec{a}\in \F_p^{d+1}:  a_0+a_1n_i+\hdots + a_dn_i^d \in \F_p^\times~\forall i\}}{(1-1/p)^t}\right).
\end{equation*}
Second, we have the cancellation 
\begin{equation*}
        \E_{f \in \mathcal{P}(d,H)}\prod_{i=1}^t \lambda(f(n_i)) = O_{t,d,\delta,A}((\log H)^{-A}).
     \end{equation*}
\end{thm}

\begin{proof}
We start with the Liouville statement since it is simpler.  Write $f(x)=a_dx^d+\cdots a_0$, where $a_0, \ldots, a_d$ range over the integers in $[-H,H]$.  Then we can express
$$\E_{f \in \mathcal{P}(d,H)}\prod_{i=1}^t \lambda(f(n_i))=\frac{1}{(2H+1)^{d+1}}\sum_{\vec a \in [-H,H]^{d+1}} \prod_{i=1}^t \lambda(\psi_i(\vec a)),$$
where the homogeneous linear forms $\psi_i=\dot \psi_i$ are given by
$$\psi_i(\vec a):=a_0+n_i a_1+\cdots+n_i^d a_d.$$
Since the $n_i$'s are distinct, the non-vanishing of the Vandermonde determinant guarantees that $\dot \psi_i$'s are in fact jointly linearly independent (not merely pairwise linearly independent).  The coefficients of the linear forms have size at most $X^d=(\log H)^{\delta^{-1}d}$, where we view $\delta^{-1} d$ as a constant.  Thus, for $A \geq \delta^{-1}d$, Theorem~\ref{thm:james-strengthened} tells us that
$$\E_{f \in \mathcal{P}(d,H)}\prod_{i=1}^t \lambda(f(n_i))=O_{t,d,\delta,A}((\log H)^{-A}),$$
as desired. 

We now turn to the von Mangoldt statement.  Without loss of generality, we may take $f_0$ to be of the form $f_0(x)=b_dx^d+\cdots+b_0$ for some integers $1 \leq b_0, \ldots, b_d \leq M$.  Then we can express $f$ as
$$f(x)=(b_d+Ma_d)x^d+\cdots+(b_0+Ma_0),$$
where each $a_j$ ranges over the integers in $I_j=I_j(H, f_0\bmod M) := \left[\frac{-H-b_j}{M},\frac{H-b_j}{M}\right]$.  Thus we can express
$$\E_{\substack{f \in \mathcal{P}(d,H)\\ f\equiv f_0 \bmod M}}\prod_{i=1}^t \Lambda(f(n_i))=(1+O(M/H))\left(\frac{M}{2H}\right)^{d+1}\sum_{\vec a \in \prod_{j=0}^d I_j} \prod_{i=1}^t \Lambda(\psi_i(\vec a)),$$
where now our linear forms are given by
$$\psi_i(\vec a):=Ma_0+Mn_i a_1+\cdots+Mn_i^d a_d+(b_0+b_1 n_i+\cdots+b_d n_i^d).$$
Again the homogeneous parts are linearly independent by Vandermonde, and the coefficients are of size $\ll MX^d=(\log H)^{A+\delta^{-1}d}$.  Applying Theorem~\ref{thm:james-strengthened} with $N=2H/M$ (say) and some $A' \geq A+\delta^{-1}d$, we find that
$$\E_{\substack{f \in \mathcal{P}(d,H)\\ f\equiv f_0 \bmod M}}\prod_{i=1}^t \Lambda(f(n_i))=(1+O(M/H))\fS_{\Psi}+O_{t,d,\delta,A}((\log H)^{-A})=\fS_{\Psi}+O_{t,d,\delta,A}((\log H)^{-A}).$$
It remains only to observe that the singular series $\fS_{\vec n, f_0 \bmod M}$ is a rewriting of the singular series $\fS_{\Psi}$ from Theorem~\ref{thm:james-strengthened}.
\end{proof}

\section{Averaged polynomial Chowla}\label{sec: poly chowla} 
The goal of this section is to prove Theorem~\ref{thm: sign patterns}.  To illustrate our methods in a simpler setting, we start with a proof of Theorem~\ref{thm: moments of polynomial Chowla}.

\subsection{Liouville sums}

For the proof of Theorem~\ref{thm: moments of polynomial Chowla}, we need the following combinatorial fact about Gaussian moments (proven, e.g., in \cite{Wilson}). 
\begin{lemma}\label{lem: Gaussian moment coefficient}
Let $k\in \N$. Then as $X\rightarrow\infty$, we have
    \begin{equation*}
        \frac{1}{X^{k/2}}\sum_{u=1}^k \sum_{\substack{\ell_1,\hdots,\ell_u\geq 1 \\ 2\mid \ell_i \\ \sum_{i=1}^u \ell_i = k}} \frac{k!}{\ell_1! \cdots \ell_u! } \sum_{1\leq n_1< \hdots < n_u\leq X}1 = C_k + O_k(X^{-1}).
    \end{equation*}
\end{lemma}

\begin{proof}[Proof of Theorem \ref{thm: moments of polynomial Chowla}]
    Fix $d\geq 1$. We expand the quantity of interest as
    \begin{equation*}
        \E_{f \in \mathcal{P}(d,H)} \left(\frac{1}{X^{1/2}}\sum_{n\leq X} \lambda(f(n))\right)^k = \E_{f \in \mathcal{P}(d,H)} \frac{1}{X^{k/2}}\sum_{u=1}^k \sum_{\substack{\ell_1,\hdots,\ell_u\geq 1\\ \sum_{i=1}^u \ell_i = k}} \frac{k!}{\ell_1!\cdots \ell_u! }\sum_{1\leq n_1< \hdots < n_u\leq X} \prod_{i=1}^u \lambda(f(n_i))^{\ell_i}.
    \end{equation*}
Since $\lambda(n)^2=1$ for all $n$, we care only about the parities of the $\ell_i$'s.  We split the sum according to whether or not there is some odd $\ell_i$: Write
\begin{equation*}
    \E_{f \in \mathcal{P}(d,H)}\left(\frac{1}{X^{1/2}}\sum_{n\leq X} \lambda(f(n))\right)^k = M_k(X;H)+ \sum_{v=1}^k \sum_{\substack{\ell_1,\hdots,\ell_v\geq 1 \\ 2\nmid \ell_i \\ \sum_{i=1}^v \ell_i \leq k}} E_k(\ell_1,\hdots,\ell_v;X;H),
\end{equation*}
where the main term is
\begin{equation*}
    M_k(X;H):= \E_{f \in \mathcal{P}(d,H)}\frac{1}{X^{k/2}}\sum_{u=1}^k \sum_{\substack{\ell_1,\hdots,\ell_u\geq 1 \\ 2\mid \ell_i \\ \sum_{i=1}^u \ell_i = k}} \frac{k!}{\ell_1!\cdots \ell_u!} \sum_{1\leq n_1< \hdots < n_u\leq X}1, 
\end{equation*}
and for for odd $\ell_1, \ldots, \ell_v$ we have defined the error term $E_k(\ell_1,\hdots,\ell_v;X;H)$ to be
\begin{equation*}
     \E_{f \in \mathcal{P}(d,H)} \frac{1}{X^{k/2}}\sum_{v'=1}^{k-v} \sum_{\substack{\ell_1',\hdots,\ell_{v'}'\geq 1 \\ \sum_{i=1}^{v'} \ell_i' + \sum_{j=1}^v \ell_j = k \\ 2\mid \ell_i'}}\frac{k!}{\ell_1!\cdots \ell_v! \cdot \ell_1'! \cdots \ell_{v'}!} \sum_{1\leq n_1< \hdots < n_v \leq X} \prod_{i=1}^v \lambda(f(n_i)).
\end{equation*}
Lemma \ref{lem: Gaussian moment coefficient} tells us that 
\begin{equation*}
    M_k(X;H) = C_k + O_k(X^{-1}),
\end{equation*}
and we will be done once we bound the error terms (of which there are only $O_k(1)$) as
\begin{equation*}
    E_k(\ell_1,\hdots,\ell_v;X;H) \ll_{d,\delta,k} X^{-1}. 
\end{equation*}

The main idea is that we can swap the order of summation and then apply Theorem~\ref{thm:GT-input} to the resulting averages over $f$.  To this end, write
\begin{equation*}
    E_k(\ell_1,\hdots,\ell_v;X;H) = \frac{1}{X^{k/2}} \sum_{v'=1}^{k-v} \sum_{\substack{\ell'_1,\hdots,\ell_{v'}'\geq 1 \\ \sum \ell_i' + \sum \ell_j = k \\ 2\mid \ell_i'}} \frac{k!}{\ell_1! \cdots \ell_v! \cdot \ell_1'! \cdots \ell_{v'}'!} \sum_{1\leq n_1< \hdots < n_v \leq X} \E_{f \in \mathcal{P}(d,H)} \prod_{i=1}^{v} \lambda(f(n_i)).
\end{equation*}
Due to the choice of $H$ and the upper bound on the $n_i$'s, Theorem \ref{thm:GT-input} lets us control the inner averages over $f$ as
$$\E_{f \in \mathcal{P}(d,H)} \prod_{i=1}^{v} \lambda(f(n_i))\ll_{d,\delta,v,A} (\log H)^{-A},$$
which is $\ll_{d,\delta,k,B}X^{-B}$ (by choosing $A$ suitably depending on $\delta,B$).  Summing over the at most $X^v$ possibilities for the $n_i$'s and the $O_k(1)$ choices for the remaining parameters, we conclude that each
$$E_k(\ell_1,\hdots,\ell_v;X;H) \ll_{d,\delta,k,B} X^{-k/2+v-B},$$
which is safely $\ll_{d,\delta,k} X^{-1}$ for a suitable choice of $B$ (e.g., $B=k+1$).  This completes the proof.

\end{proof}

\subsection{Sign patterns}
The proof of Theorem~\ref{thm: sign patterns} follows the same general strategy as the proof of Theorem~\ref{thm: moments of polynomial Chowla}, but we must replace Lemma~\ref{lem: Gaussian moment coefficient} with more complicated counting arguments.  The connection goes as follows.  It is straightforward to show that the equidistribution of Liouville sign patterns is equivalent to Chowla's Conjecture for polynomials of the form $f(x)=(x+a_1) \cdots (x+a_t)$; more precisely, the equidistribution of Liouville sign patterns of length $s$ is equivalent to Chowla's Conjecture for all polynomials $f(x)=(x+a_1) \cdots (x+a_t)$ with $1 \leq a_1<\cdots<a_t \leq s$.  The forward implication is immediate, and the backward implication follows from the important observation that for any $(\epsilon_1, \ldots, \epsilon_s) \in \{-1,1\}^s$, we have
\begin{equation}\label{eq:sign-patterns-FOIL}
2^{-s} \prod_{i=1}^s (1+\epsilon_i \lambda(n_i))=\begin{cases}
1, &\text{if } \lambda(n_i)=\epsilon_i
 \text{ for all $i$};\\
0, &\text{otherwise}.
 \end{cases}
\end{equation}
In the setting of Theorem~\ref{thm: sign patterns}, this identity implies that
\begin{equation}\label{eq:sign-patterns-foil-outcome}
\#\{n \leq X: (\lambda(f(n+1)), \ldots, \lambda(f(n+s)))=(\epsilon_1, \ldots, \epsilon_s)\}=2^{-s} \sum_{n \leq X} \prod_{i=1}^s (1+\epsilon_i \lambda(f(n+i)));
\end{equation}
the advantage of this maneuver is that expanding out the product and averaging over $f$ gives a sum of several expressions, each of the type that we analyzed in the previous subsection.

\begin{proof}[Proof of Theorem~\ref{thm: sign patterns}]
We wish to estimate the moments of
\begin{equation}\label{eq:sign-patterns-main-quantity}
\frac{\#\{n \leq X: (\lambda(f(n+1)), \ldots, \lambda(f(n+s)))=(\epsilon_1, \ldots, \epsilon_s)\}-2^{-s}X}{X^{1/2}}
\end{equation}
as $f$ ranges over $\mathcal{P}(d,H)$.  Expanding the product in \eqref{eq:sign-patterns-foil-outcome}, we can express this quantity as
$$\frac{2^{-s}}{X^{1/2}}\sum_{n \leq X}\sum_{\emptyset \neq T \subseteq [s]} \prod_{i \in T} \epsilon_i \lambda(f(n+i)));$$
notice that $T=\emptyset$ corresponds to the ``main term'' $2^{-s}X$ that we subtracted off.  Now the $k$-th moment of this random variable is
\begin{equation}\label{eq:sign-pattern-moment-expansion}
\left(\frac{2^{-s}}{X^{1/2}}\right)^k \E_{f \in \mathcal{P}(d,H)} \left( \sum_{n \leq X} \sum_{\emptyset \neq T \subseteq [s]}\prod_{i \in T} \epsilon_i \lambda(f(n+i))) \right)^k,
\end{equation}
which we can expand as
$$\left(\frac{2^{-s}}{X^{1/2}}\right)^k\sum_{n_1, \ldots, n_k \leq X} \sum_{\emptyset \neq T_1, \ldots, T_k \subseteq [s]} \E_{f \in \mathcal{P}(d,H)} \prod_{(i,j): i \in T_j} \epsilon_i \lambda(f(n_j+i)).$$
For each fixed choice of $n_1, \ldots, n_k, T_1, \ldots, T_k$, we will estimate the contribution of the term
$$\prod_{(i,j): i \in T_j} \epsilon_i \cdot \E_{f \in \mathcal{P}(d,H)} \prod_{(i,j): i \in T_j} \lambda(f(n_j+i)).$$
As in the proof of Theorem~\ref{thm: moments of polynomial Chowla}, whether this contribution goes towards a main term or towards an error term depends on the form of the multiset $\{n_j+i: i \in T_j\}$.  Say that a multiset is \emph{even} if all of its elements appear with even multiplicity.  Theorem~\ref{thm:GT-input} tells us that
$$\E_{f \in \mathcal{P}(d,H)} \prod_{(i,j): i \in T_j} \lambda(f(n_j+i))=\begin{cases}
1, &\text{if the multiset $\{n_j+i: i \in T_j\}$ is even;}\\
O_{d,\delta,s,k,A}((\log H)^{-A}), &\text{otherwise}.
\end{cases}$$
Since there are at most $X^k 2^{sk}$ choices of $n_1, \ldots, n_k, T_1, \ldots, T_k$, the total contribution to \eqref{eq:sign-pattern-moment-expansion} of the non-even multisets is
$$\ll_{d,\delta, s, k, A}\left(\frac{2^{-s}}{X^{1/2}}\right)^k X^k 2^{sk} \cdot (\log H)^{-A},$$
which is comfortably $O_{d,\delta,s,k}(X^{-1})$ if we choose $A$ sufficiently large in terms of $\delta,k$.  We are left with the contribution of the even multisets, and we conclude that \eqref{eq:sign-pattern-moment-expansion} is equal to 
$$\left(\frac{2^{-s}}{X^{1/2}}\right)^k\sum_{\emptyset \neq T_1, \ldots, T_k \subseteq [s]} \left(\prod_{(i,j): i \in T_j} \epsilon_i \right) \cdot  \Bigl|\Bigl\{ 
\substack{\displaystyle (n_1, \ldots, n_k): \text{ the multiset}\\ \displaystyle \{n_j+i: i \in T_j\} \text{ is even}}
\Bigr\}\Bigr|+O_{d,\delta,s,k}(X^{-1}).$$ 

We pause to consider the analogous problem with $\lambda(f(1)), \lambda(f(2)), \ldots$ replaced by a sequence of independently random $\pm 1$'s.  Let $y(1), y(2), \ldots \in \{-1,1\}$ be chosen independently and uniformly at random.  The same argument that led to \eqref{eq:sign-pattern-moment-expansion} tells us that the $k$-th moment of the quantity
\begin{equation}\label{eq:random-sequence-comparison}
\frac{\#\{n \leq X: (y(n+1), \ldots, y(n+s))=(\epsilon_1, \ldots, \epsilon_s)\}-2^{-s}X}{X^{1/2}}
\end{equation}
is precisely
$$\left(\frac{2^{-s}}{X^{1/2}}\right)^k \E_{\vec{y}} \left( \sum_{n \leq X} \sum_{\emptyset \neq T \subseteq [s]}\prod_{i \in T} \epsilon_i y(n+i)) \right)^k,$$
where $\E_{\vec y}$ denotes the expectation over $\vec y=(y(1), \ldots, y({X+s}))$ drawn uniformly at random from $\{-1,1\}^{X+s}$
Expanding and swapping the order of summation gives
$$\left(\frac{2^{-s}}{X^{1/2}}\right)^k\sum_{n_1, \ldots, n_k \leq X} \sum_{\emptyset \neq T_1, \ldots, T_k \subseteq [s]} \prod_{(i,j):i\in T_j} \epsilon_i \cdot \E_{\vec{y}} \prod_{(i,j): i \in T_j} y(n_j+i).$$
The inner expectation is $1$ when the multiset $\{n_j+i: i \in T_j\}$ is even, and it is $0$ otherwise, so we can rewrite this $k$-th moment as
\begin{equation}\label{eq:sign-moments-clean}
\left(\frac{2^{-s}}{X^{1/2}}\right)^k\sum_{\emptyset \neq T_1, \ldots, T_k \subseteq [s]} \left(\prod_{(i,j): i \in T_j} \epsilon_i \right) \cdot  \Bigl|\Bigl\{ 
\substack{\displaystyle (n_1, \ldots, n_k): \text{ the multiset}\\ \displaystyle \{n_j+i: i \in T_j\} \text{ is even}}
\Bigr\}\Bigr|,
\end{equation}
which agrees (up to $+o(1)$) with what appeared in the Liouville setting.

Thus, if one of the quantities in \eqref{eq:sign-patterns-main-quantity}, \eqref{eq:random-sequence-comparison} has a limiting distribution as $X \to \infty$ that is determined by its moments, then these two quantities have the same limiting distribution.  Indeed, it is a consequence of the so-called ``$m$-dependent central limit theorem'' of Hoeffding and Robbins \cite{hoeffding1994central} that the limiting distribution of the quantity in \eqref{eq:random-sequence-comparison} is a Gaussian with mean zero; since Gaussians are determined by their moments, it follows that the quantity in \eqref{eq:sign-patterns-main-quantity} has the same limiting distribution.  The result of Hoeffding and Robbins also provides a way to compute the variance of this Gaussian in terms of certain covariances, and one can conclude the proof of Theorem~\ref{thm: sign patterns} by appealing to their work.  For the sake of completeness and simplicity, we will also provide a direct combinatorial argument for estimating the moments in \eqref{eq:sign-moments-clean}.

It will be convenient to use some language from graph theory for bookkeeping.  Fix a choice of $T_1, \ldots, T_k$, and consider the complete multipartite graph $G=G(T_1, \ldots, T_k)$ with partite sets $T_1, \ldots, T_k$.  For each choice of $n_1, \ldots, n_k$, we can obtain a subgraph $G'=G'(T_1, \ldots, T_k, n_1, \ldots, n_k)$ of $G$ by keeping the edge between the vertex $i_1$ in part $T_{j_1}$ and the vertex $i_2$ in part $T_{j_2}$ if and only if $n_{j_1}+i_1=n_{j_2}+i_2$.  In this language, the multiset $\{n_j+i: i \in T_j\}$ is even if and only if the graph $G'$ has a perfect matching.  Such a perfect matching can exist only when $|T_1|+\cdots+|T_k|$ is even, so we restrict our attention to choices of $T_1, \ldots, T_k$ for which this is the case.

Fix a subgraph $G^*$ of $G$ that contains a perfect matching.  We will estimate the number of tuples $(n_1, \ldots, n_k)$ for which $G'(T_1, \ldots, T_k, n_1, \ldots, n_k)=G^*$.  Consider the auxiliary graph $J$ with the vertex set $[k]$, where vertices $j_1, j_2$ are adjacent if and only if $G^*$ contains an edge connecting $T_{j_1}$ and $T_{j_2}$.  It is clear that $J$ has no isolated vertices (due to $G^*$ containing a perfect matching), so $J$ has some number $c \leq k/2$ of connected components.  The connected components of $J$ partition $[k]$.  For any connected component of $J$ with vertex set $U \subseteq [k]$, the corresponding edges of $G^*$ place linear constraints on $\{n_j: j \in U\}$, and specifying the value of $n_j$ for any single $j \in U$ determines the value of $n_j$ for every other $j \in U$; thus there are at most $X$ possible choices for the tuple $(n_j)_{j \in U}$.  In total, the number of tuples $(n_1, \ldots, n_j)$ with $G'=G^*$ is at most $X^c$.  There are $O_{s,k}(1)$ possibilities for graphs $G,G^*$, so the contribution to \eqref{eq:sign-moments-clean} of the $G^*$'s with $c<k/2$ is
$$\ll_{s,k}\left(\frac{2^{-s}}{X^{1/2}}\right)^k  X^{k/2-1} \ll_{s,k}(X^{-1}).$$

It remains to consider the contribution of the graphs $G^*$ such that the auxiliary graph $J$ has exactly $c=k/2$ connected components.  In this case, the graph $J$ must be a perfect matching, and for each each edge $(j_1,j_2)$ in $J$, the graph $G^*$ must contain a perfect matching between the vertex sets $T_{j_1}, T_{j_2}$.  This is possible exactly when the sets $T_{j_1}, T_{j_2}$ are translates of one another, in which case there are $X-O_s(1)$ choices for the pair $(n_{j_1}, n_{j_2})$.  (The corresponding edges of $G^*$ match the $\ell$-th smallest element of $T_{j_1}$ with the $\ell$-th smallest element of $T_{j_2}$ for all $\ell$.)  This gives $(X-O_s(1))^{k/s}$ choices for $(n_1, \ldots, n_s)$ such that $G'$ contains $G^*$, and by the considerations of the previous paragraph the contribution where $G'$ properly contains $G^*$ is negligible.   Thus,  the contribution to \eqref{eq:sign-moments-clean} of each such $G^*$ is
$$\left(\frac{2^{-s}}{X^{1/2}}\right)^k \left(\prod_{(i,j): i \in T_j} \epsilon_i \right) \cdot (X-O_s(1))^{k/2}+O_{s,k}(X^{-1})=2^{-sk} \left(\prod_{(i,j): i \in T_j} \epsilon_i \right)+O_{s,k}(X^{-1}).$$

The total number of possibilities for the graph $J$ is the number of perfect matchings of $[k]$, namely, $C_k=(k-1)!!$.  Consider the contribution to \eqref{eq:sign-moments-clean} of all of the choices of $T_1, \ldots T_k$ corresponding to a given perfect matching $J$; this is independent of $J$ (since the elements of $[k]$ can be relabeled), so we can estimate \eqref{eq:sign-moments-clean} as $C_k$ times the contribution from a single fixed $J$, say, the graph $J_0$ with the edges $(1,2), (3,4), \ldots, (k-1,k)$.  The sum over such tuples $(T_1, \ldots, T_k)$ splits as a product over the edges of $J_0$, and we can write \eqref{eq:sign-moments-clean} as 
$$C_k \cdot 2^{-sk} \cdot \left( \sum_{\substack{\emptyset \neq T_1,T_2 \subseteq [s]\\ \text{translates}}} \prod_{i \in T_1} \epsilon_i \cdot \prod_{i \in T_2} \epsilon_i  \right)^{k/2}+O_{s,k}(X^{-1}),\vspace{-.25cm}$$
as desired.
\end{proof}

\section{Moments for averaged Bateman--Horn}\label{sec:BH-moments}
In this section we establish Theorem \ref{thm: moments of BH}. 
 The methods are similar to those of the previous section, but the computation is considerably more involved due to the presence of a main term. Recall that $H = \exp(X^\delta)$ for some $0<\delta<1$.

 \begin{proof}[Proof of Theorem~\ref{thm: moments of BH}]
Set $w=w(X):=(\log X)/(\log\log X)$.  We want to bound the $k$-th moment
\begin{equation*}
    \E_{f \in \mathcal{P}(d,H)} \left(\frac{1}{X}\sum_{n\leq X} \Lambda(f(n))-\fS_f(w)\right)^k  = \E_{f \in \mathcal{P}(d,H)}\sum_{j=0}^k (-1)^j \binom{k}{j} \fS_f(w)^{k-j} X^{-j} \sum_{n_1,\hdots,n_j\leq X} \prod_{i=1}^j \Lambda(f(n_i)).
\end{equation*}
Switching the order of summation, we can rewrite this as
\begin{equation*}
 \sum_{j=0}^k (-1)^j \binom{k}{j} X^{-j} S_j(X;H),
\end{equation*}
where for each $0\leq j\leq k$ we have set
\begin{equation*}
    S_j(X;H) := \sum_{n_1,\hdots,n_j\leq X} \E_{f \in \mathcal{P}(d,H)} \fS_f(w)^{k-j} \prod_{i=1}^j \Lambda(f(n_i)).
\end{equation*}
Our intermediate goal is to show that for each $j\geq 1$, the sum $S_j(X;H)$ is dominated by the contribution where $n_1, \ldots, n_j$ are all distinct.  This distinctness will then facilitate the application of our main tool Theorem~\ref{thm:GT-input}.  The final step of the argument consists of exhibiting cancellation in the resulting sums of singular series.  The details are as follows.

For $1\leq \ell \leq j$, consider the contribution
\begin{equation*}
    S_{j,\ell}(X;H) := \sum_{\substack{n_1,\hdots,n_j\leq X \\ \#\{n_1,\hdots,n_j\}=\ell}} \E_{f \in \mathcal{P}(d,H)} \fS_f(w)^{k-j} \prod_{i=1}^j \Lambda(f(n_i))
\end{equation*}
from tuples $(n_1, \ldots, n_j)$ containing exactly $\ell$ distinct numbers.  There are at most $X^\ell$ choices for the \emph{set} $\{n_1, \ldots, n_j\}$, and each such set corresponds to $O_k(1)$ tuples $(n_1, \ldots, n_j)$, so in total we are summing over $O_k(X^\ell)$ tuples.  Lemma~\ref{lem: upper bound on truncated singular series} tells us that for any $\varepsilon>0$, we have $\fS_f(w) \ll_{d,k,\varepsilon}X^{\varepsilon/k}$ for all $f \in \mathcal{P}(d,H)$, so certainly $\fS_f(w)^{k-j} \ll_{d,k,\varepsilon}X^{\varepsilon}$.  Finally, we have
\begin{equation*}
    \Lambda(f(n)) \leq \log(2HX^d) \leq X^{\delta}+d\log X+\log(2)\ll_{d,\delta} X^\delta
\end{equation*}
for all $f \in \mathcal{P}(d,H)$ and $n \leq X$.  In total we have
$$S_{j,\ell}(X,H) \ll_{d,\delta,k,\varepsilon}X^{\ell+\varepsilon+j\delta}.$$
If $\ell<j$, then this exponent is at most $j-1+\varepsilon+k\delta$; since $k\delta<1$ by assumption, we can choose $\varepsilon>0$ sufficiently small (depending on $\delta,k$) that the exponent is strictly smaller than $j$, which gives
$$S_{j,\ell}(X;H) \ll_{d,\delta,k}X^{j-\gamma}$$
for some $\gamma=\gamma(\delta,k)>0$.

Our next goal is to evaluate $S_{j,j}(X;H)$ using Theorem \ref{thm:GT-input}. Fix a tuple $\Vec{n} = (n_1,\hdots,n_j)$ of distinct positive integers of size at most $X$, and consider the average
\begin{equation*}
    S_j(X,H;\Vec{n}) := \E_{f \in \mathcal{P}(d,H)} \fS_f(w)^{k-j} \prod_{i=1}^j \Lambda(f(n_i)).
\end{equation*}
Again splitting the expectation modulo $P=P(w)$ (see Section \ref{subsec: notation}), we obtain
\begin{equation*}
    S_j(X,H;\Vec{n}) = \E_{f_0 \bmod{P}}(1+O_d(P/H)) \cdot \fS_{f_0}(w)^{k-j} \E_{\substack{f \in \mathcal{P}(d,H)\\ f\equiv f_0 \bmod P}} \prod_{i=1}^j \Lambda(f(n_i)).
\end{equation*}
From $P\sim \exp(\log X /\log\log X)\ll X$ we see that $P,n_1,\hdots,n_j\ll_{\delta} \log(H)^{A'}$ for some $A'>0$, so we can use Theorem \ref{thm:GT-input} (with $M=P$) to evaluate the inner expectation.  We obtain
\begin{align*}
    S_{j}(X,H;\Vec{n}) &= \E_{f_0 \bmod{P}}(1+O_d(P/H)) \fS_{f_0}(w)^{k-j} (\fS_{\Vec{n},f_0\bmod P} + O_{d,k,A}((\log H)^{-2A}))\\
    &= \E_{f_0 \bmod{P}} \fS_{f_0}(w)^{k-j} \fS_{\Vec{n},f_0 \bmod P} + O_{d,k,A}((\log H)^{-A}),
\end{align*}
where we used Lemma \ref{lem: upper bound on truncated singular series} to pass to the second line. Putting everything together gives
\begin{align*}
    S_{j,j}(X;H) &= \sum_{\substack{n_1,\hdots,n_j\leq X \\ \textrm{distinct}}} S_j(X,H;\Vec{n}) \\
    &= \sum_{\substack{n_1,\hdots,n_j\leq X \\ \textrm{distinct}}} \E_{f_0 \bmod{P}}\fS_{f_0}(w)^{k-j}\fS_{\Vec{n},f_0\bmod P} + O_{d,k,A}(X^j\log(H)^{-A}).
\end{align*}
Since $\log(H)\gg X^{\delta}$, a suitable choice of $A$ lets us bound the error term by $O_{d,\delta,k}(X^{j-1})$. 

We now turn to the main term of $S_{j,j}(X;H)$.  We truncate the singular series $\fS_{\Vec{n},f_0\bmod P}$ and split the $n_i$'s into arithmetic progressions modulo $P$.  
The key to the calculation will be that the average over $\vec{n} \bmod{P}$ of the truncation of $\fS_{\Vec{n},f_0\bmod P}$ precisely matches $\fS_{f_0}(w)^j$.  Since this singular series $\fS_{\Vec{n},f_0\bmod P}$ converges absolutely, we may safely write 
\begin{equation*}
    \fS_{\Vec{n},f_0\bmod P} =(1+O(w^{-1})) \prod_{p\leq w} \prod_{i=1}^j \frac{\mathbf{1}_{f_0(n_i)\in \F_p^\times}}{1-1/p};
\end{equation*}
notice that the product depends only on $\vec{n} \bmod{P}$.  So our main term from above can be written as
\begin{equation*}
    (1+O(w^{-1})) \E_{f_0 \bmod P} \sum_{\Vec{m} \bmod P} \fS_{f_0}(w)^{k-j}  \prod_{p\leq w} \prod_{i=1}^j \frac{\mathbf{1}_{f_0(m_i)\in \F_p^\times}}{1-1/p} \sum_{\substack{n_1,\hdots,n_j\leq X \\ \textrm{distinct}\\ \Vec{n}\equiv \Vec{m}\bmod P}}1.
\end{equation*}
For each choice of $\vec{m}$, the number of $\vec n$'s in the sum is $$(X/P)^j+O_j(X^{j-1})=(1+O_k((\log X)^{-1}))(X/P)^j$$
(recall $P \sim \exp(\log X/\log\log X) \ll_\varepsilon X^\varepsilon$ for all $\varepsilon>0$), and we can simplify our main term as
\begin{equation*}
    (1+O_k(w^{-1}))X^j \E_{f_0 \bmod P}\E_{\Vec{m}\bmod P} \fS_{f_0}(w)^{k-j} \prod_{p\leq w} \prod_{i=1}^j \frac{\mathbf{1}_{f_0(m_i)\in \F_p^\times}}{1-1/p}. 
\end{equation*}
Here we have again used (the argument from) Lemma \ref{lem: upper bound on truncated singular series} to upper-bound $\fS_{f_0}(w)$ for the error term. 
Bringing the average over $\vec m$ inside the product over $p,i$, we find that
$$\E_{\Vec{m}\bmod P} \prod_{p\leq w} \prod_{i=1}^j \frac{\mathbf{1}_{f_0(m_i)\in \F_p^\times}}{1-1/p}=\prod_{p\leq w} \prod_{i=1}^j \frac{p^{-1}\#\{m_i\bmod p: f_0(m_i)\in \F_p^\times\}}{1-1/p}=\fS_{f_0}(w)^j.$$
Reinstating the average over $f_0$, we conclude that

\begin{equation*}
    S_{j,j}(X;H) = (1+O(w^{-1})) X^j \E_{f_0\bmod P} \fS_{f_0}(w)^{k} + O_{d,\delta,k}(X^{j-1/2})
\end{equation*}
for $j \geq 1$. Observe that $X^{j-1/2} = O(w^{-1})\cdot X^j \E_{f_0\bmod P}\fS_{f_0}(w)^k$, so we can drop the second error term here.

The only remaining term, corresponding to $j=0$, is
\begin{equation*}
    S_0(X;H) = \E_{f \in \mathcal{P}(d,H)}\fS_{f}(w)^{k}.
\end{equation*}
Since $P\ll_{\varepsilon} X^\varepsilon$ for all $\varepsilon>0$, we easily obtain
\begin{equation*}
    S_0(X;H) = \E_{f_0\bmod P} \fS_{f_0}(w)^{k} + O_\varepsilon(H^{-1}X^{\varepsilon}).
\end{equation*}
Recall that $H =\exp(X^{\delta})$, so this error is much smaller than $O(X^{-1/2}).$

Putting everything together, we have
\begin{align*}
\E_{f \in \mathcal{P}(d,H)} &\left(\frac{1}{X}\sum_{n\leq X}\Lambda(f(n))-\fS_f(w)\right)^k  = \sum_{j=0}^k (-1)^j \binom{k}{j} X^{-j} S_j(X;H)    \\
&= \sum_{j=0}^k (-1)^j \binom{k}{j} (1+O(w^{-1})) \E_{f_0\bmod P} \fS_{f_0}(w)^{k} + O_{d,\delta,k,A}(X^{-1/2}+(\log X)^{-A}).
\end{align*} 
The main terms cancel perfectly due to the identity $\sum_{j=0}^k (-1)^j \binom{k}{j}=0$, and the remaining error terms have size
\begin{equation*}
    \ll_{d,\delta,k} w^{-1} \E_{f_0\bmod P} \fS_{f_0}(w)^k+(\log X)^{-1} .
\end{equation*}
Lemma \ref{lem: upper bound on truncated singular series} gives $\fS_{f_0}(w)\ll \log w$ pointwise, and (with plenty of room to spare) we obtain the desired upper bound $O_{d,\delta,k}(w ^{-1+o(1)}) = O_{d,\delta,k}((\log X)^{-1+o(1)})$.
\end{proof}

\section{Prime tuples on average}
\label{sec:tuples}

Our goal is to prove the following refinement of Theorem \ref{thm: BH prime tuple hybrid}; the latter corresponds to the special case $w(X)=(\log X)/(\log \log X)$. 
\begin{thm}\label{thm: refined GOAL BH prime tuple hybrid}
Let $d \geq 1$ be an integer and let $0<\delta<1$ be a real.  Set $H=H(X):=\exp(X^\delta)$, and suppose that $w=w(X)\leq (\log X)/(\log\log X)$.  Let $1 \leq k<1/\delta$ be an integer, and let $\alpha>0$ be a real.  Then for any distinct integers $\ell_1,\ldots,\ell_k$ of size at most $X^{\alpha}$, we have
    $$\E_{f \in \mathcal{P}(d,H)}\left|\frac{1}{X}\sum_{n\leq X} \prod_{i=1}^k \Lambda(f(n+\ell_i)) - \fS_{f,\Vec{\ell}}\,(w)\right|^2 = O_{d,\delta,\alpha,k}(w^{-1+o_k(1)}).$$
\end{thm}

\begin{proof}
Unsurprisingly, we first expand the square as  \begin{multline*}
    \E_{f \in \mathcal{P}(d,H)} \left|\frac{1}{X}\sum_{n\leq X} \prod_{i=1}^k \Lambda(f(n+\ell_i)) - \fS_{f,\Vec{\ell}}\,(w)\right|^2  
    = \E_{f \in \mathcal{P}(d,H)} \frac{1}{X^2} \sum_{n_1,n_2\leq X} \prod_{i=1}^k \Lambda(f(n_1+\ell_i))\Lambda(f(n_2+\ell_i)) \\
    -\E_{f \in \mathcal{P}(d,H)} \frac{2}{X} \fS_{f,\Vec{\ell}}\,(w)  \sum_{n\leq X} \prod_{i=1}^k \Lambda(f(n+\ell_i)) + \E_{f \in \mathcal{P}(d,H)} \fS_{f,\Vec{\ell}}\,(w)^2.
\end{multline*}
We will see that the second term is approximately twice the size of the last term. Again, set $P:=P(w)$ and recall that $\fS_{f,\Vec{\ell}}\,(w)$ depends only on $f\bmod P$; with this in mind, we will split the expectation over $f$ into residue classes modulo $P$.  Recall that $P\ll \exp(\log X/\log\log X)=X^{o(1)}$. Hence an application of Theorem \ref{thm:GT-input} gives
\begin{multline*}
    \frac{2}{X}\E_{f \in \mathcal{P}(d,H)}  \fS_{f,\Vec{\ell}}\,(w) \sum_{n\leq X} \prod_{i=1}^k \Lambda(f(n+\ell_i)) \\ = \frac{2}{X}\E_{f_0\bmod P} \fS_{f_0,\Vec{\ell}}\,(w) \sum_{n\leq X} \E_{\substack{f\in \mathcal{P}(d,H) \\ f\equiv f_0 \bmod P}} \prod_{i=1}^\ell \Lambda(f(n+\ell_i)) + O_d(P/H)\\
    = \frac{2}{X} \E_{f_0\bmod P} \fS_{f_0,\Vec{\ell}}\,(w) \sum_{n\leq X} \fS_{n+\Vec{\ell},f_0\bmod P} + O_{d,\delta,\alpha,k,A}(\log(H)^{-A}).
\end{multline*}
Since $H = \exp(X^{\delta})$, we can bound the error term by $O_{d,\delta,\alpha,k}(X^{-1})$. 

The next task is evaluating the sum of singular series. Recall that we may truncate $\fS_{n+\Vec{\ell},f_0\bmod P}$ since the singular series converges absolutely. In particular, we have
\begin{equation*}
    \fS_{n+\Vec{\ell}, f_0 \bmod P} = (1+O_k(w^{-1})) \prod_{p\leq w} \prod_{i=1}^k \frac{\mathbf{1}_{f_0(n+\ell_i)\in \F_p^\times}}{1-1/p},
\end{equation*}
where (for fixed $f_0$) the product depends only on $n\bmod P$. Hence
\begin{multline*}
    \frac{2}{X} \E_{f_0\bmod P} \fS_{f_0,\Vec{\ell}}\,(w) \sum_{n\leq X} \fS_{n+\Vec{\ell},f_0\bmod P} \\
    = \frac{2(1+O_k(w^{-1}))}{P^{d+1}} \sum_{f_0,n_0\bmod P} \fS_{f_0,\Vec{\ell}}\,(w) \prod_{p\leq w}  \prod_{i=1}^k \frac{\mathbf{1}_{f_0(n+\ell_i)\in \F_p^\times}}{1-1/p} \cdot \frac{1}{X} \sum_{\substack{n\leq X \\ n\equiv n_0 \bmod P}}1 \\ 
    = \frac{2(1+O_k(w^{-1}))}{P^{d+1}} \sum_{f_0\bmod P} \fS_{f_0,\Vec{\ell}}\,(w) \prod_{p\leq w} \left(\frac{p^{-1}\#\{n\in \F_p: f_0(n+\ell_i)\in \F_p^\times ~\forall i\}}{(1-1/p)^k}\right) + O_k(X^{-9/10})\\
    = 2(1+O_k(w^{-1})) \E_{f_0\bmod P}\fS_{f_0,\Vec{\ell}}\,(w)^2 + O_k(X^{-9/10}).
\end{multline*}
In obtaining the error term we used the fact that $P \ll_\varepsilon
X^\varepsilon$ for all $\varepsilon>0$ and the pointwise upper bound on $\fS_{f_0,\Vec{\ell}}\,(w)$ from Lemma \ref{lem: upper bound on truncated singular series}.  At the same time, 
\begin{equation*}
    \E_{f \in \mathcal{P}(d,H)}\fS_{f,\Vec{\ell}}\,(w)^2 =(1+O_d(P/H)) \E_{f_0\bmod P} \fS_{f_0,\Vec{\ell}}\,(w)^2 \E_{\substack{f\in \mathcal{P}(d,H)\\ f\equiv f_0\bmod P}} 1 
    = \E_{f_0\bmod P} \fS_{f_0,\Vec{\ell}}\,(w)^2 + O(H^{-1/2}).
\end{equation*}
Thus, inserting the pointwise bound $\fS_{f_0,\Vec{\ell}}\,(w)\ll (\log w)^k =w^{o_k(1)}$ from Lemma \ref{lem: upper bound on truncated singular series}, we find that the quantity in the statement of Theorem \ref{thm: refined GOAL BH prime tuple hybrid} is equal to
\begin{equation*}
    \frac{1}{X^2}\E_{f \in \mathcal{P}(d,H)} \sum_{n_1,n_2\leq X} \prod_{i=1}^k \Lambda(f(n_1+\ell_i)) \Lambda(f(n_2+\ell_i)) - \E_{f_0\bmod P}\fS_{f_0,\Vec{\ell}}\,(w)^2 + O(w^{-1+o_k(1)});
\end{equation*}
it remains to bound this quantity

We separate the diagonal and off-diagonal contributions in the first sum.  The diagonal terms are the pairs $(n_1,n_2)$ satisfying
\begin{equation*}
    \{n_1+\ell_i: 1\leq i\leq k\} \cap \{n_2+\ell_i: 1\leq i\leq k\} \neq \emptyset.
\end{equation*}
We can bound the diagonal contribution by
\begin{equation*}
    \frac{1}{X^2}\sum_{j=1}^k \log(HX^d)^j \sum_{\substack{n_1,n_2\leq X \\ |(n_1+\Vec{\ell})\cap (n_2+\Vec{\ell})| = j}} \E_{f \in \mathcal{P}(d,H)} \prod_{i=1}^k \Lambda(f(n_1+\ell_i)) \prod_{n_2+\ell_i\not\in n_1+\Vec{\ell}} \Lambda(f(n_2+\ell_i)). 
\end{equation*}
Using either Theorem \ref{thm:GT-input} or an upper-bound sieve, we can bound the latter by 
\begin{equation}\label{eq: hybrid bound on diagonal}
    \ll_{k,\varepsilon} \frac{1}{X^2} \sum_{j=1}^k \log(HX^d)^j \sum_{\substack{n_1,n_2\leq X \\ |(n_1+\Vec{\ell})\cap (n_2+\Vec{\ell})| = j}}  X^{\varepsilon} \ll_{\delta,k} X^{\delta k - 1 + \varepsilon}.
\end{equation}
Since $\delta<1/k$, the diagonal contribution is $\ll_{k,\varepsilon'} X^{\varepsilon'}$.

The off-diagonal terms are those with $(n_1+\Vec{\ell})\cap (n_2+\Vec{\ell}) = \emptyset$. In this case, a direct application of Theorem \ref{thm:GT-input} gives
\begin{multline*}
    \frac{1}{X^2} \sum_{\substack{n_1,n_2\leq X \\(n_1+\Vec{\ell})\cap (n_2+\Vec{\ell}) = \emptyset }} \E_{f \in \mathcal{P}(d,H)}\prod_{i=1}^k \Lambda(f(n_1+\ell_i)) \Lambda(f(n_2+\ell_i)) \\
    = \frac{1}{X^2} \sum_{\substack{n_1,n_2\leq X \\(n_1+\Vec{\ell})\cap (n_2+\Vec{\ell}) = \emptyset }} \fS_{n_1+\Vec{\ell},n_2+\Vec{\ell}}\, + O_{d,k,A}((\log H)^{-A}),
\end{multline*}
where the singular series is given by
\begin{equation*}
    \fS_{n_1+\Vec{\ell},n_2+\Vec{\ell_2}}:= \prod_{p} \left(\frac{p^{-d-1}\#\{f_0\in \F_p[x]: \deg(f_0) \leq d; f_0(n_1+\ell_i),f_0(n_2+\ell_i)\in \F_p^\times ~\forall i\}\}}{(1-1/p)^{2k}}\right).
\end{equation*}
As $d\geq 1$, this singular series converges absolutely, and truncating it at $w$ yields the estimate
\begin{equation*}
    \fS_{n_1+\Vec{\ell},n_2+\Vec{\ell}} = (1+O(w^{-1})) \prod_{p\leq w} \frac{1}{p^{d+1}}\sum_{f_0\bmod p} ~\prod_{i=1}^k \frac{\mathbf{1}_{f_0(n_1+\ell_i)\in \F_p^\times}}{(1-1/p)}\cdot\frac{\mathbf{1}_{f_0(n_2+\ell_i)\in \F_p^\times}}{(1-1/p)}.
\end{equation*}
This product depends only on $n_1,n_2\bmod P$, so the main term of the off-diagonal contribution can be written as
\begin{multline*}
     (1+O(w^{-1})) \frac{1}{P^{d+1}} \sum_{f_0\bmod P} \sum_{n_1',n_2'\bmod P} \prod_{p\leq w} \prod_{i=1}^k \frac{\mathbf{1}_{f_0(n_1'+\ell_i)\in \F_p^\times}}{(1-1/p)}\frac{\mathbf{1}_{f_0(n_2'+\ell_i)\in \F_p^\times}}{(1-1/p)} \times \frac{1}{X^2}\sum_{\substack{n_1,n_2\leq X \\ n_i\equiv n_i'\bmod P \\ (n_1+\Vec{\ell})\cap (n_2+\Vec{\ell}) = \emptyset}}1.
\end{multline*}
Since $P\ll_\varepsilon X^\varepsilon$ for all $\varepsilon>0$, this inner sum is 
\begin{equation*}
    \sum_{\substack{n_1,n_2\leq X \\ n_i\equiv n_i'\bmod P \\ (n_1+\Vec{\ell})\cap (n_2+\Vec{\ell}) = \emptyset}}1 = \left(\frac{X}{P}\right)^2 + O_k(X),
\end{equation*}
and the off-diagonal main term is
\begin{multline*}
    (1+O(w^{-1})) \frac{1}{P^{d+1}}\sum_{f_0\bmod P} \prod_{p\leq w} \frac{1}{p^2}\sum_{n_1',n_2'\bmod p} \prod_{i=1}^k \frac{\mathbf{1}_{f_0(n_1'+\ell_i)\in \F_p^\times}}{(1-1/p)}\frac{\mathbf{1}_{f_0(n_2'+\ell_i)\in \F_p^\times}}{(1-1/p)}  + O_k(X^{-1/2}) \\ 
    = (1+O(w^{-1})) \frac{1}{P^{d+1}}\sum_{f_0\bmod P} \prod_{p\leq w} \left(\frac{p^{-1}\#\{x\in\F_p: f_0(x+\ell_i)\in \F_p^\times~ \forall i\}}{(1-1/p)^k}\right)^2 + O_k(X^{-1/2}) \\
    = (1+O(w^{-1})) \frac{1}{P^{d+1}}\sum_{f_0\bmod P} \fS_{f_0,\Vec{\ell}}\,(w)^2 + O_k(X^{-1/2}). 
\end{multline*}

Let us return to the main quantity in the statement of Theorem \ref{thm: refined GOAL BH prime tuple hybrid}. Combining the above calculation with \eqref{eq: hybrid bound on diagonal} and recalling that $\fS_{f_0,\Vec{\ell}}\,(w)\ll \log w$, we see that this main quantity is bounded by 
\begin{multline*}
    \left|\frac{1}{X^2}\E_{f \in \mathcal{P}(d,H)} \sum_{n_1,n_2\leq X} \prod_{i=1}^k \Lambda(f(n_1+\ell_i)) \Lambda(f(n_2+\ell_i)) - \E_{f_0\bmod P}\fS_{f_0,\Vec{\ell}}\,(w)^2\right| + O_k(w^{-1+o_k(1)}) \\ \ll_\varepsilon w^{-1+o_k(1)} + O(w^{-1})\left|\E_{f_0\bmod P} \fS_{f_0,\Vec{\ell}}\,(w)^2 \right| + O_{k,d,\varepsilon,A}( X^{\delta k - 1 +\varepsilon} + X^{-1/2}+ \log(H)^{-A})  \\
    \ll_{k,d,\varepsilon} X^{\delta k - 1 + \varepsilon} + w^{-1+o_k(1)}.
\end{multline*}
We conclude on noting that the first error term is negligible since $\delta< 1/k$.
\end{proof}

We remark that this argument allows us to bound not just the second moment but also all higher moments.  We presented only the second-moment argument here because it will suffice for our applications in the following section.

\section{Distributions for prime tails}\label{sec:tail}
We model our proof of our distributional results Theorems \ref{thm: Poisson distribution} and \ref{thm:primes-in-intervals-gaussian} on Gallagher's proof of Theorem \ref{thm: Gallagher}.  The main input is Theorem \ref{thm: refined GOAL BH prime tuple hybrid}, our averaged prime tuples result. 

\subsection{From prime tuples to moments}

We can use our prime tuples result Theorem~\ref{thm: BH prime tuple hybrid} to deduce information on the typical moments of $\rho_{f,X,L}$ for various ranges of $L$.  To this end, set
$$M_k(\rho_{f,X,L}):=\E_{Y \sim \rho_{f,X,L}}[Y^k]=X^{-1}\sum_{n \leq X}(\pi_f(n+L)-\pi_f(n))^k,$$
where $\pi_f(n)$ denotes the number of $1 \leq x \leq n$ such that $f(x)$ is prime.  Let $S(k,r)$ denote the Stirling number of the second kind (the number of partitions of $[k]$ into exactly $r$ parts).

\begin{thm}\label{thm:M}
Let $d \geq 1$, let $0<\alpha<2/3$, and let $\delta(X)=\delta_{d,\alpha}(X)$ be the function from Remark~\ref{rem: diagonalization}.  Set $H=H(X):=\exp(X^{\delta(X)})$.  Let $L=L(X) \in \mathbb{N}$ satisfy $0< L\leq X^\alpha$ and $L(X)\rightarrow\infty$ as $X\rightarrow\infty$. Let $w=w(X) \in \mathbb{N}$ satisfy $w(X)\leq \min(\log L(X)/\log\log L(X), \log X/\log \log X)$.  Then for every fixed $k \in \mathbb{N}$, we have
$$\E_{f \in \mathcal{P}(d,H)} \left|M_{k}(\rho_{f,X,L}) -  \sum_{r=1}^k S(k,r)\left(\frac{L\fS_f(w)}{\log(HX^d)}\right)^r\right|^2 \ll_{d,\alpha,k} \frac{1}{w^{1-o_k(1)}} \sum_{r=1}^k \left(\frac{L}{\log(HX^d)}\right)^{2r}.$$
\end{thm}

We remark that our argument would work with any $\alpha<1$, and we have presented the proof with $\alpha<2/3$ for notational simplicity; this makes no difference in our applications, where $L(X)$ will in fact grow more slowly than any positive power of $X$.

Theorem~\ref{thm:M} shows that for most $f$'s, the quantity $M_{k}(\rho_{f,X,L})$ is very close to the $k$-th moment of a Poisson random variable of expectation $\mathcal{L}:=\frac{L\fS_f(w)}{\log(HX^d)}$.  This underpins Theorem~\ref{thm: Poisson distribution} in the regime of constant $\mathcal{L}$.  In the regime where $\mathcal{L}$ tends to infinity, we deduce Theorem~\ref{thm:primes-in-intervals-gaussian} from the standard fact that a Poisson distribution with growing expectation limits to a Gaussian after suitable centering and rescaling.

We isolate two steps of the proof in advance.  The first lemma records a singular series identity; this is an important component of the argument.  The second lemma lets us replace $\Lambda$ with the indicator function of the primes in the conclusion of Theorem~\ref{thm: BH prime tuple hybrid}; this maneuver is standard.  
\begin{lemma}\label{lem:singular-series-for-M-moments}
Let $d \geq 1$ and let $\alpha>0$. Let $L=L(X) \in \mathbb{N}$ satisfy $0< L\leq X^\alpha$ and $L(X)\to\infty$ as $X\to\infty$. Let $w=w(X) \in \mathbb{N}$ satisfy $w(X)\leq \min(\log L(X)/\log\log L(X), \log X/\log \log X)$. Then for every $f\in \Z[x]$, irreducible and of degree $d$, and every $r \in \mathbb{N}$, we have
$$ (L\fS_f(w))^r = \sum_{\substack{1\leq \ell_1,...,\ell_r\leq L \\ \text{distinct}}} \fS_{f,\vec \ell}\,(w) + O_r(L^{r-1+o_{r,\alpha}(1)}).$$
\end{lemma}

\begin{proof}
We start with the sum on the right-hand side. Again, set $P:= P(w)$.  We have
\begin{multline*}
   \sum_{\substack{1\leq \ell_1,...,\ell_r\leq L \\\text{distinct}}}\fS_{f,\Vec{\ell}}\,(w) = \sum_{\substack{1\leq \ell_1,...,\ell_r\leq L \\\text{distinct}}} \prod_{p\leq w} \frac{p^{-1} \#\{x\in \F_p: f(x+\ell_i)\in \F_p^\times ~\forall i\}}{(1-1/p)^r}  \\
   = \sum_{\Vec{l}\in (\Z/P\Z)^r} \prod_{p\leq w} \frac{p^{-1} \#\{x\in \F_p: f(x+l_i)\in \F_p^\times~ \forall i\}}{(1-1/p)^r} \\ \times \#\{1\leq \ell_1,...,\ell_r\leq L: \ell_i \equiv l_i \bmod P,~ \text{distinct}\}.
\end{multline*}
For each $0 
\leq j<P$, let $e_j$ denote the number of indices $i$ such that $l_i=j$. Then we know that 
\begin{multline*}
    \#\{1\leq \ell_1,...,\ell_r\leq L: \ell_i \equiv l_i \bmod P,~ \text{distinct}\} = \prod_{j} \frac{\lfloor (L-j)/P \rfloor ! }{\lfloor (L-j)/P - e_j\rfloor!} =\left(\frac{L}{P}\right)^r + O_r(L^{r-1}).
\end{multline*}
Hence, we get that 
\begin{multline}\label{eq: summing singular series over tuples}
    \sum_{\substack{1\leq \ell_1,...,\ell_r\leq L \\\text{distinct}}}\fS_{f,\Vec{\ell}}\,(w) = \left(\frac{L}{P}\right)^r\sum_{\Vec{l}\in (\Z/P\Z)^r} \prod_{p\leq w} \frac{p^{-1} \#\{x\in \F_p: f(x+l_i)\in \F_p^\times ~\forall i\}}{(1-1/p)^r} \\ + O_r\left(L^{r-1} \sum_{\Vec{l}\in (\Z/P\Z)^r} \fS_{f,\Vec{l}}\,(w)\right).  
\end{multline}
Lemma \ref{lem: upper bound on truncated singular series} tells us that $\fS_{f,\Vec{l}}\,(w) \ll (\log w)^r$, so the error term is $O_r(L^{r-1}P^r (\log w)^r)$. From $w \leq  (\log L)/(\log\log L)$ we get $P=L^{o_\alpha(1)}$, and the error can be bounded by $O(L^{r-1+o_{r,\alpha}(1)})$. 

Next, we manipulate the singular series
\begin{equation*}
    \sum_{\Vec{l}\in (\Z/P\Z)^r} \prod_{p\leq w} \frac{p^{-1} \#\{x\in \F_p: f(x+l_i)\in \F_p^\times ~\forall i\}}{(1-1/p)^r} = \prod_{p\leq w} \sum_{\Vec{l}\in \F_p^r} \frac{p^{-1}\#\{x\in \F_p: f(x+l_i)\in \F_p^\times ~\forall i\}}{(1-1/p)^r}.
\end{equation*}
Switching the order of summation, we see that 
\begin{align*}
    \sum_{\Vec{l}\in \F_p^r} \#\{x\in \F_p: f(x+l_i)\in \F_p^\times~ \forall i\} &= \sum_{x\in \F_p}~ \prod_{i=1}^r \#\{l_i\in \F_p: f(x+\ell_i) \in \F_p^\times\} \\
    &= p\cdot \#\{y\in \F_p: f(y)\in \F_p^\times\}^r.
\end{align*}
Thus we can write
\begin{equation*}
    \prod_{p\leq w} \sum_{\Vec{l}\in \F_p^r} \frac{p^{-1}\#\{x\in \F_p: f(x+l_i)\in \F_p^\times ~\forall i\}}{(1-1/p)^r} = \prod_{p\leq w} \left(\frac{\#\{y\in \F_p: f(y)\in \F_p^\times\}}{1-1/p}\right)^r.
\end{equation*}
The main term in \eqref{eq: summing singular series over tuples} now becomes 
\begin{equation*}
    L^r \prod_{p\leq w}\left(\frac{p^{-1}\#\{y\in \F_p: f(y) \in \F_p^\times\}}{1-1/p}\right)^r = L^r \fS_{f}(w)^r,
\end{equation*}
and we conclude that 
\begin{equation*}
    \sum_{\substack{1\leq \ell_1,...,\ell_r\leq L \\\text{distinct}}}\fS_{f,\Vec{\ell}}\,(w) = (L\fS_{f}(w))^r + O_r(L^{r-1+o_{r,\alpha}(1)}). \qedhere
\end{equation*}
\end{proof}

\begin{lemma}\label{lem:von-mangoldt-to-primes}
Let $d\geq 1$ and let $0<\delta<1$ be a real. Set $H=H(X) := \exp(X^\delta)$. Let $1\leq k< 1/3\delta$ be an integer and let $0<\alpha<2/3$ be a real. Then for any distinct integers $\ell_1, \ldots, \ell_k$ of size at most $X^\alpha$, we have
$$\E_{f \in \mathcal{P}(d,H)} \left|\#\{n\leq X: f(n+\ell_i)\in \mathscr{P} ~\forall i\} - X\frac{ \fS_{f,\vec \ell}\,(w)}{\log(HX^d)^k }\right|^2 \ll_{d,\delta, \alpha,k} \frac{X^2 w^{-1+o_k(1)}}{\log(HX^d)^{2k}}.$$
\end{lemma}

\begin{proof}
Theorem~\ref{thm: refined GOAL BH prime tuple hybrid} tells us that
\begin{equation*}
    \E_{f \in \mathcal{P}(d,H)}\left|\frac{1}{X}\sum_{n\leq X}\prod_{i=1}^k \Lambda(f(n+\ell_i)) - \fS_{f,\Vec{\ell}}\,(w)\right|^2 \ll_{d,\delta,\alpha,k}w^{-1+o_k(1)}.
\end{equation*}
This will give us the desired control once we replace $\Lambda$ with the function
\begin{equation*}
    \theta(n) := \begin{cases}
        \log(|n|), & n\text{ prime}; \\ 
        0, & \text{otherwise}
    \end{cases}
\end{equation*}
and apply partial summation.

First, we claim that 
\begin{equation}\label{eq: hybrid with thetas}
    \E_{f \in \mathcal{P}(d,H)}\left|\frac{1}{X}\sum_{n\leq X}\prod_{i=1}^k \theta(f(n+\ell_i)) -  \fS_{f,\Vec{\ell}}\,(w)\right|^2 \ll_{d,\delta,\alpha,k} w^{-1+o_k(1)}.
\end{equation}
Observe that 
\begin{multline*}
    \E_{f\in \calP(d,H)} \left|\frac{1}{X}\sum_{n\leq X} \prod_{i=1}^k \theta(f(n+\ell_i)) - \fS_{f,\Vec{\ell}}\,(w)\right|^2 \\ \ll \E_{f\in \calP(d,H)}\left|\frac{1}{X}\sum_{n\leq X} \prod_{i=1}^k \theta(f(n+\ell_i)) - \Lambda(f(n+\ell_i))\right|^2 + \left|\frac{1}{X}\sum_{n\leq X} \prod_{i=1}^k \Lambda(f(n+\ell_i)) - \fS_{f,\Vec{\ell}}\,(w)\right|^2.
\end{multline*}
By Theorem ~\ref{thm: refined GOAL BH prime tuple hybrid}, the second term is $O_{d,\delta,\alpha,k}(w^{-1+o(1)})$. For the first term, we crudely bound
\begin{multline*}
    \E_{f\in \calP(d,H)}\left|\frac{1}{X}\sum_{n\leq X} \prod_{i=1}^k \theta(f(n+\ell_i)) - \Lambda(f(n+\ell_i))\right|^2 \\ \ll \log(HX^d)^k \E_{f\in \calP(d,H)} \left|\frac{1}{X}\sum_{n\leq X} \prod_{i=1}^k \theta(f(n+\ell_i)) - \Lambda(f(n+\ell_i))\right|.
\end{multline*}
Telescoping the product and applying the Triangle Inequality gives 
\begin{multline*}
    \E_{f\in \calP(d,H)} \left|\frac{1}{X}\sum_{n\leq X} \left(\prod_{i=1}^k \theta(f(n+\ell_i)) - \prod_{i=1}^k \Lambda(f(n+\ell_i)\right)\right|\\ \leq \sum_{j=1}^k \E_{f\in \calP(d,H)} \frac{1}{X} \sum_{n\leq X} \left|\prod_{i=1}^{j-1} \theta(f(n+\ell_i)) \times (\theta(f(n+\ell_j))-\Lambda(f(n+\ell_j)) \times \prod_{i=j+1}^k \Lambda(f(n+\ell_i))\right|. 
\end{multline*}
The functions $\theta,\Lambda$ differ only on proper prime powers (that is, numbers of the form $p^r$ with $p$ prime and $r >1$), so the quantity inside the absolute value is nonzero only when $f(n+\ell_j)$ is a proper prime power; when this does occur it can be bounded by $\ll_{d,\delta,k} (\log H)^k$.  For each fixed $n$, the quantity $f(n+\ell_j)$ is a perfect $r$-th power for at most a $O(H^{-1+1/r})$-proportion of $f$'s; summing over $r \geq 2$, we conclude that $f(n+\ell_j)$ is a proper prime power for at most a $O(H^{-1/2})$-proportion of $f$'s. Thus the previous centered equation is at most $\ll_{d,\delta,k} (\log H)^{k}H^{-1/2}$, which is much smaller than $w^{-1+o_k(1)}$.
This completes the proof of \eqref{eq: hybrid with thetas}.

With \eqref{eq: hybrid with thetas} in hand, we return to our main quantity of interest, namely, 
\begin{equation*}
    \E_{f\in \calP(d,H)} \left|\#\{n\leq X: f(n+\ell_i)\in \scrP ~\forall i\} - X \frac{\fS_{f,\Vec{\ell}}\,(w)}{\log(HX^d)^k}\right|^2.
\end{equation*}
We can bound this by
\begin{multline*}
\ll\E_{f\in \calP(d,H)} \left|\#\{n\leq X: f(n+\ell_i)\in \mathscr{P}~\forall i\} - \frac{1}{\log(HX^d)^k} \sum_{n\leq X} \prod_{i=1}^k \theta(f(n+\ell_i))\right|^2 \\ + \E_{f\in \calP(d,H)} \left|\frac{1}{\log(HX^d)^k} \left(\sum_{n\leq X} \prod_{i=1}^k \theta(f(n+\ell_i)) - X\fS_{f,\Vec{\ell}}\,(w)\right)\right|^2,
\end{multline*}
where \eqref{eq: hybrid with thetas} tells us that the second term is $\ll_{d,\delta,\alpha,k} \frac{w^{-1+o_k(1)}X^2}{\log(HX^d)^{2k}}$.
We will handle the first term using partial summation. Let us split the expectation over $f\in \calP(d,H)$ according to whether $|a_d|\geq H/X^{2/3}$ or $|a_d|\leq H/X^{2/3}$. When $|a_d|\leq H/X^{2/3}$, we bound the sum over $n$ trivially as
\begin{multline*}
    \frac{1}{(2H+1)^{d+1}}\sum_{\substack{f \in \mathcal{P}(d,H)\\ |a_d|\leq H/X^{2/3}}}\left|\#\{n\leq X: f(n+\ell_i)\in \mathscr{P}~\forall i\} - \frac{1}{\log(HX^d)^k} \sum_{n\leq X} \prod_{i=1}^k \theta(f(n+\ell_i))\right|^2\\
     \ll \frac{1}{(2H+1)^{d+1}}\sum_{\substack{f \in \mathcal{P}(d,H)\\ |a_d|\leq H/X^{2/3}}} O(X^2) \ll X^{2-2/3}.
\end{multline*}
The assumption $k<1/3\delta$ ensures that the power of $X$ in $(\log (HX^d))^{2k} \ll_d X^{2k\delta}$ is strictly smaller than $2/3$, so our $X^{2/3}$ saving is acceptable for the conclusion of the lemma.

We now turn to the contribution of $f\in \calP(d,H)$ with $|a_d|\geq H/X^{2/3}$. Observe that if $t\geq 2X^{2/3}$ then $|f(t)| \asymp |a_d|t^d$, i.e., $|f(t)|$ can be small only when $t\leq 2 X^{2/3}$.  We thus split our sum 
\begin{multline*}
    \frac{1}{(2H+1)^{d+1}}\sum_{\substack{f \in \mathcal{P}(d,H)\\ |a_d|> H/X^{2/3}}}\left|\#\{n\leq X: f(n+\ell_i)\in \mathscr{P}~\forall i\} - \frac{1}{\log(HX^d)^k} \sum_{n\leq X} \prod_{i=1}^k \theta(f(n+\ell_i))\right|^2
\end{multline*}
according to the size of $n$. This quantity is $\ll S_1+S_2$, where $S_1$ is the small-$n$ contribution
$$\frac{1}{(2H+1)^{d+1}}\sum_{\substack{f \in \mathcal{P}(d,H)\\ |a_d|> H/X^{2/3}}}\left|\#\{n\leq 2X^{2/3}: f(n+\ell_i)\in \mathscr{P}~\forall i\} - \frac{1}{\log(HX^d)^k} \sum_{n\leq 2X^{2/3}} \prod_{i=1}^k \theta(f(n+\ell_i))\right|^2$$
and $S_2$ is the large-$n$ contribution
$$\frac{1}{(2H+1)^{d+1}}\sum_{\substack{f \in \mathcal{P}(d,H)\\ |a_d|> H/X^{2/3}}}\left|\#\{n\in [2X^{2/3},X]: f(n+\ell_i)\in \mathscr{P}~\forall i\} - \frac{1}{\log(HX^d)^k} \sum_{n\in [2X^{2/3},X]} \prod_{i=1}^k \theta(f(n+\ell_i))\right|^2.$$
We can trivially dispose of the first term as
\begin{equation*}
    S_1 \ll \frac{1}{(2H+1)^{d+1}}\sum_{\substack{f \in \mathcal{P}(d,H)\\ |a_d|> H/X^{2/3}}} X^{4/3} \ll X^{2-2/3},
\end{equation*}
which is acceptable (as described at the end of the previous paragraph).  We use partial summation to handle $S_2$.  Writing
$$\psi_{f,\Vec{\ell}}\,(t):= \sum_{n\leq X}\prod_{i=1}^k \theta(f(n+\ell_i)),$$
we find by partial summation that
\begin{multline*}
    \#\{n\in [2X^{2/3},X]: f(n+\ell_i)\in \scrP~\forall i\} = \frac{1}{\log(HX^d)^k} \sum_{n\in [2X^{2/3},X]} \prod_{i=1}^k \theta(f(n+\ell_i)) \\
    + O\left(\int_{2X^{2/3}}^X \psi_{f,\Vec{\ell}}\,(t) \cdot \frac{d}{dt}\left(\frac{1}{\prod_{i=1}^k \log(f(t+\ell_i))}\right)\,dt\right).
\end{multline*}
The lower bound on $t$ (together with the upper bound on the $\ell_i$'s) ensures that
\begin{equation*}
    \frac{d}{dt}\left(\frac{1}{\prod_{i=1}^k\log(f(t+\ell_i))}\right) \ll_k \frac{1}{t (\log |f(t)|)^{k+1}}.
\end{equation*}
Putting everything together, we have
\begin{multline*}
    S_2\ll_k \frac{1}{(2H+1)^{d+1}}\sum_{\substack{f \in \mathcal{P}(d,H)\\ |a_d|> H/X^\beta}} \left|\int_{2X^{2/3}}^X \frac{\psi_{f,\Vec{\ell}}\,(t)}{t(\log|f(t)|)^{k+1}} \,dt\right|^2\\
    \ll_k \frac{1}{\log(HX^d)^{2(k+1)}} \int_{2X^{2/3}}^X \int_{2X^{2/3}}^X \E_{f\in \calP(d,H)} \frac{\psi_{f,\Vec{\ell}}\,(t)}{t}\cdot \frac{\psi_{f,\Vec{\ell}}\,(t')}{t'}\,dt \,dt'.
\end{multline*}
Since $\delta<1/2k$, it is straightforward to check (by Cauchy--Schwarz) that the integrand is uniformly $O_k(1)$, so again this error term is acceptable (with room to spare).
\end{proof}

With these two ingredients in hand, we can prove Theorem~\ref{thm:M}.

\begin{proof}[Proof of Theorem~\ref{thm:M}]
Since $\pi_f(n+L)-\pi_f(n)=\sum_{\ell=1}^L {\bf 1}_{f(n+\ell) \in \mathscr{P}}$, we can expand
\begin{multline*}
    M_{k}(\rho_{f,X,L}) = \sum_{n\leq X} (\pi_{f}(n+L)-\pi_f(n))^k 
    = \sum_{r=1}^k S(k,r) \sum_{\substack{1\leq \ell_1,...,\ell_r\leq L\\ \text{distinct}}} \#\{n\leq X: f(n+\ell_i)\in \mathscr{P} ~ \forall i\}.
\end{multline*}
The main idea is that the sums over the $\ell_i$'s are of the form treated in Theorem~\ref{thm: BH prime tuple hybrid} (via Lemma~\ref{lem:von-mangoldt-to-primes}).  Applying the Cauchy--Schwarz Inequality, inserting Lemma~\ref{lem:singular-series-for-M-moments}, and applying the Cauchy--Schwar Inequality again, we get 
\begin{multline*}
\left|M_{k}(\rho_{f,X,L}) -  X\sum_{r=1}^k S(k,r) \left(\frac{L\fS_{f}(w)}{\log(HX^d)}\right)^r\right|^2 \\ \leq \sum_{r=1}^kS(k,r)^2\cdot\sum_{r=1}^k \left|\sum_{\substack{1\leq \ell_1,...,\ell_r\leq L\\ \text{distinct}}} \#\{n\leq X: f(n+\ell_i)\in \mathscr{P} ~\forall i\} -X \left(\frac{L\fS_f(w)}{\log(HX^d)}\right)^r\right|^2\\
    \ll_k \sum_{r=1}^k \left|\sum_{\substack{1\leq \ell_1,...,\ell_r\leq L \\ \text{distinct}}} \left(\#\{n\leq X: f(n+\ell_i)\in \mathscr{P} ~\forall i\} -X\frac{ \fS_{f,\vec \ell}\,(w)}{\log(HX^d)^r}+O\left(X\frac{L^{-1+o_{r,\alpha}(1)}}{\log (HX^d)^r}\right)\right)\right|^2 \\
    \leq  \sum_{r=1}^k L^r\sum_{\substack{1\leq \ell_1,...,\ell_r\leq L \\\text{distinct}}} \left|\#\{n\leq X: f(n+\ell_i)\in \mathscr{P} ~\forall i\} -X\frac{ \fS_{f,\vec \ell}\,(w)}{\log(HX^d)^r}+O\left(X\frac{L^{-1+o_{r,\alpha}(1)}}{\log (HX^d)^r}\right)\right|^2\\
    \ll \sum_{r=1}^k L^r\sum_{\substack{1\leq \ell_1,...,\ell_r\leq L \\\text{distinct}}} \left[\left|\#\{n\leq X: f(n+\ell_i)\in \mathscr{P} ~\forall i\} -X\frac{ \fS_{f,\vec \ell}\,(w)}{\log(HX^d)^r}\right|^2+O\left(X\frac{L^{-1+o_{r,\alpha}(1)}}{\log (HX^d)^r}\right)^2 \right].
\end{multline*}
We now average over $f \in \mathcal{P}(d,H)$.  Lemma~\ref{lem:von-mangoldt-to-primes} lets us bound the the first term's contribution by
$$\ll_{d,\delta,\alpha,k} \sum_{r=1}^k L^r\sum_{\substack{1\leq \ell_1,...,\ell_r\leq L \\\text{distinct}}} \frac{X^2w^{-1+o_k(1)}}{\log(HX^d)^{2r}} \ll_{d,\alpha,k} \frac{X^2}{w^{1-o_k(1)}} \sum_{r=1}^k \left(\frac{L}{\log(HX^d)}\right)^{2r},$$
as desired.
The second term contributes 
$$\ll_k \sum_{r=1}^k L^{2r} \left(X \frac{L^{-1+o_{r,\alpha}(1)}}{\log(HX^d)^r} \right)^2,$$
which is certainly acceptable.
\end{proof}

\subsection{Poisson tail}
Theorem~\ref{thm:M} tells us that, for a suitable range of $L$, the moments $M_k$ typically match the moments of a Poisson distribution.  The Poisson distribution, like the Gaussian, is determined by its moments.  However, before we can deduce Theorem~\ref{thm: Poisson distribution}, we must address three further points: We need to filter out $f$'s with $\fS_f=0$, we need a flexible version of Theorem~\ref{thm:M} that allows $L$ to depend on $f$, and we need to obtain convergence for all moments simultaneously.  The most important of these is the second: Since the value of $L_i$ depends on $f_i$ in the statement of Theorem~\ref{thm: Poisson distribution}, we need a statement like Theorem~\ref{thm:M} for many different values of $L$.  Taking a union bound over all of the possible values of $L$ (corresponding to the possible polynomials $f$) is much too expensive, so instead we will apply Theorem~\ref{thm:M} with a suitable ``net'' of $L$-values.

\begin{proof}[Proof of Theorem~\ref{thm: Poisson distribution}]
For now, fix some $k \in \mathbb{N}$. Lemmas~\ref{lem: lower bound truncated singular series} and~\ref{lem: upper bound on truncated singular series} provide a constant $C=C(d)>0$ such that
$$C^{-1}(\log w)^{1-d}\leq \fS_f(w) \leq C \log w$$
whenever the singular series is nonzero.  These upper and lower bounds differ by a multiplicative factor of $D=D(X):=C^2(\log w)^d$.  Now let $\tau=\tau(X)$ be another slowly-growing function; we will later take $\tau$ to grow slightly faster than $\log D$.  For $0 \leq j \leq \tau$, consider the parameter
$$L(j):=\frac{\mathcal{L}\log(HX^d)}{D^{j/\tau}C^{-1}(\log w)^{1-d}}.$$
The $L(j)$'s are decreasing from $\frac{\mathcal{L}\log(HX^d)}{C^{-1} (\log w)^{1-d}}$ to $\frac{\mathcal{L}\log(HX^d)}{C\log w}$, and consecutive $L(j)$'s have quotient $L(j)/L(j+1)=D^{1/\tau}=1+O((\log D)/\tau)$.

Theorem~\ref{thm:M} with $$w=w(X):=\frac{\delta(X)\log X}{(\log\log X)^2} \leq \frac{\log L(\tau)}{\log\log L(\tau)} \leq \frac{\log L(j)}{\log\log L(j)}$$ gives
\begin{multline*}
\E_{f \in \mathcal{P}(d,H)} \left|M_{k}(\rho_{f,X,L(j)}) -  \sum_{r=1}^k S(k,r)\left(\frac{L(j)\fS_f(w)}{\log(HX^d)}\right)^r\right|^2 \ll_k \frac{1}{w^{1-o(1)}} \sum_{r=1}^k \left(\frac{L(j)}{\log(HX^d)}\right)^{2r}\\
 \ll_k \frac{1}{w^{1-o_k(1)}} \left(1+ \left(\mathcal{L} \cdot C (\log w)^{d-1}\right)^{2k} \right)
 \ll_{k,\mathcal{L}} \frac{1}{w^{1-o_k(1)}} (\log w)^{2k(d-1)}.
\end{multline*}
Markov's Inequality and a union bound over $j$ give
\begin{multline*}
\mathbb{P}_{f \in \mathcal{P}(d,H)}\left[\left|M_{k}(\rho_{f,X,L(j)}) -  \sum_{r=1}^k S(k,r)\left(\frac{L(j)\fS_f(w)}{\log(HX^d)}\right)^r\right| \geq \varepsilon \text{ for some $0 \leq j \leq \tau$}\right]\\
\ll_{\mathcal{L},k} \frac{\tau\varepsilon^{-2}(\log w)^{2k(d-1)}}{w^{1-o_k(1)}},
\end{multline*}
uniformly in $\varepsilon>0$.

Now suppose we pick $f \in \mathcal{P}(d,H)$ uniformly at random subject to the constraint $\fS_f(w) \neq 0$.  Take the $0 \leq j \leq \tau-1$ satisfying
$$\frac{L(j)\fS_f(w)}{\log(HX^d)} \geq \mathcal{L}>\frac{L(j+1)\fS_f(w)}{\log(HX^d)}.$$
Since these upper and lower bounds differ by a multiplicative factor of $(1+O((\log D)/\tau))$, we conclude that they are both $(1+O((\log D)/\tau))\mathcal{L}$.  In particular, we have
$$\sum_{r=1}^k S(k,r)\left(\frac{L(j)\fS_f(w)}{\log(HX^d)}\right)^r, \quad \sum_{r=1}^k S(k,r)\left(\frac{L(j+1)\fS_f(w)}{\log(HX^d)}\right)^r=\sum_{r=1}^k S(k,r)\mathcal{L}^r+O_{\mathcal{L},k}((\log D)/\tau).$$
It follows from the inequality at the end of the previous paragraph that with probability at least
$$1-O_{d,\mathcal{L},k} \left( \frac{\tau\varepsilon^{-2}(\log w)^{2k(d-1)}}{w^{1-o_k(1)}} \right),$$
we have
$$M_{k}(\rho_{f,X,L(j)}), ~ M_{k}(\rho_{f,X,L(j+1)})=\sum_{r=1}^k S(k,r)\mathcal{L}^r+O_{\mathcal{L},k}(\varepsilon+(\log D)/\tau).$$
(For the bound on the error probability, note that $\fS_f(w) \neq 0$ for a positive proportion (depending on $d$) of polynomials $f \in \mathcal{P}(d,H)$, so our conditioning changes the error term by only a constant factor.)
Since
$$L(j+1)< \frac{\mathcal{L}\log(HX^d)}{\fS_f(w)} \leq L(j)$$
and $M_{k}(\rho_{f,X,L})$ is monotone in $L$, this implies that
\begin{align}\label{eq:poisson-moment-single-f}
M_{k}\left(\rho_{f,X;\frac{\mathcal{L}\log(HX^d)}{\fS_f(w)}}\right)=\sum_{r=1}^k S(k,r)\mathcal{L}^r+O_{\mathcal{L},k}(\varepsilon+(\log D)/\tau).
\end{align}
This is nearly the desired conclusion.  It remains to take a sequence of $f$'s and specify $\varepsilon, \tau$.

Let $X_i,H_i,w_i,f_i,L_i$ be as in the statement of Theorem~\ref{thm: Poisson distribution}.  Fix $\varepsilon>0$.  Let $E_i$ denote the event that \eqref{eq:poisson-moment-single-f} fails for $f_i$.  Then
$$\sum_{i=1}^\infty \Pr[E_i] \ll_{d,\mathcal{L},k} \varepsilon^{-2} \sum_{i=1}^\infty \frac{\tau(X_i)(\log w_i)^{2k(d-1)}}{w_i^{1-o_k(1)}}.$$
Notice that $w_i=(\log X_i)^{1-o(1)}$ since $\delta(X)$ grows slowly (in fact we can take it to grow arbitrarily slowly).   Set $\tau(X_i):=\varepsilon^{-1}\log D$, so that the error on the right-hand side of \eqref{eq:poisson-moment-single-f} is $O_{\mathcal{L},k}(\varepsilon)$ and $$\tau(X_i) \ll \varepsilon^{-1} \log\log w_i \ll \varepsilon^{-1} \log\log \log X_i.$$  Now we can bound
$$\sum_{i=1}^\infty \Pr[E_i] \ll_{d,\mathcal{L},k} \varepsilon^{-3} \sum_{i=1}^\infty \frac{(\log\log\log X_i)(\log \log X_i)^{2k(d-1)}}{(\log X_i)^{1-o_k(1)}} \ll_k \varepsilon^{-3} \sum_{i=1}^\infty \frac{1}{(\log X_i)^{1-o_k(1)}},$$
and this sum converges since the sequence $X_i$ grows sufficiently quickly.  The Borel--Canteli Lemma then tells us that with probability $1$, the event $E_i$ occurs for only finitely many $i$'s; in particular, with probability $1$ it will fail for all sufficiently large $i$.  This means that with probability $1$, we have
$$M_{k}\left(\rho_{f_i,X_i,L_i}\right)=\sum_{r=1}^k S(k,r)\mathcal{L}^r+O_{\mathcal{L},k}(\varepsilon)$$
for all sufficiently large $i$.  Since this holds for all $\varepsilon>0$, we have
$$\lim_{i \to \infty}M_{k}\left(\rho_{f_i, X_i,L_i}\right)=\sum_{r=1}^k S(k,r)\mathcal{L}^r$$
with probability $1$.  By a union bound, with probability $1$ this is the case for all $k \in \mathbb{N}$ simultaneously.

Finally, since the Poisson distribution is determined by its moments, we conclude that with probability $1$, the distribution $\rho_{f_i,X_i,L_i}$ converges to a Poisson distribution with expectation $\mathcal{L}$.  The statement of the theorem follows.
\end{proof}

\subsection{Gaussian counts}
The deduction of Theorem~\ref{thm:primes-in-intervals-gaussian} goes similarly to the deduction of Theorem~\ref{thm: Poisson distribution}.  The main difference is that now $\mathcal{L}$ is growing and we must rescale our distributions before taking a limit.  We need the following (standard) identity  relating rescaled central moments to raw moments for the Poisson distribution.  Recall that we write $C_k$ for the $k$-th moment of the Gaussian of mean $0$ and variance $1$.  We use the convention $S(0,0)=1$ and $S(0,r)=0$ for $r>0$.

\begin{lemma}\label{lem:poisson-to-gaussian-conversion}
Set $m_\ell(\lambda):=\sum_{r=0}^\ell S(\ell,r)\lambda^r$ for $\ell \in \Z_{\geq 0}$.  Then for every $k \in \N$, we have the identity
$$\frac{\sum_{\ell=0}^k \binom{k}{\ell} m_\ell(\lambda) (-\lambda)^{k-\ell}}{\lambda^{k/2}}=C_k+O_k(\lambda^{-1/2})$$
as $\lambda \to \infty$.
\end{lemma}

\begin{proof}
We proceed by induction on $k$.  For convenience, set
$$\mu_k:=\sum_{\ell=0}^k \binom{k}{\ell} m_\ell \cdot (-\lambda)^{k-\ell}$$
for $k \in \N$.  For the base cases, is easy to check that $$\mu_1=m_0 \cdot (-\lambda)+m_1=-\lambda+\lambda=0$$
and
$$\mu_2=m_0 \cdot \lambda^2+2m_1 \cdot (-\lambda)+m_2=\lambda^2-2\lambda^2+(\lambda+\lambda^2)=\lambda,$$
as desired.  We will show inductively that $\mu_k$ satisfies the recurrence
\begin{equation}\label{eq:moments-recurrence}
\mu_k=\lambda \sum_{t=0}^{k-2} \binom{k-1}{t} \mu_t.
\end{equation}
From this recurrence it follows that each $\mu_k$ is a polynomial in $\lambda$, and that the degree of this polynomial is at most $\lfloor k/2 \rfloor$.  Isolating the $t=k-2$ contribution and using the induction hypothesis $\lambda^{-(k-2)/2}\mu_{k-2}=C_{k-2}+O_k(\lambda^{-1/2})$, we find that
\begin{align*}
\lambda^{-k/2}\mu_k &=(k-1)\lambda^{1-k/2} \mu_{k-2}+O_k(\lambda^{1-k/2+\lfloor (k-3)/2 \rfloor})\\
 &=(k-1)(C_{k-2}+O_k(\lambda^{-1/2}))+O_k(\lambda^{-1/2})\\
 &=C_k+O_k(\lambda^{-1/2})
\end{align*}
(since $C_k=(k-1)C_{k-2}$ by definition).  It remains only to establish \eqref{eq:moments-recurrence}.

We need the combinatorial identity
\begin{equation}\label{eq:stein-chen}
m_{\ell+1}=\lambda \sum_{s=0}^{\ell} \binom{\ell}{s}m_s,
\end{equation}
which is just a restatement of the Stein--Chen identity; we include the short proof for the reader's convenience.  Indeed, inserting the definition of $m_s$, we see that we need to show that
$$\sum_{r=0}^{\ell+1} S(\ell+1,r)\lambda^r=\lambda \sum_{s=0}^{\ell} \binom{\ell}{s}\sum_{u=0}^{s} S(s,u)\lambda^u.$$
Gathering like powers of $\lambda$ on the right-hand side, we need to show that
$$S(\ell+1,r)=\sum_{s=r-1}^{\ell} \binom{\ell}{s}S(s,r-1)$$
for all $1 \leq r \leq \ell+1$ (we can ignore $r=0$ since neither side of the previous centered equation has any constant term).  The left-hand side counts the number of ways to partition $[\ell+1]$ into exactly $r$ parts.  The data of such a partition consists of the set $X \subseteq [\ell]$ of elements contained in the same part as the element $\ell+1$, together with a partition of $[\ell] \setminus X$ into exactly $r-1$ parts; this is precisely what the right-hand side counts.  We have now established \eqref{eq:stein-chen}.

Returning to \eqref{eq:moments-recurrence}, we can use Pascal's Identity and \eqref{eq:stein-chen} to expand
\begin{align*}
\mu_k=\sum_{\ell=0}^k \binom{k}{\ell} m_\ell \cdot (-\lambda)^{k-\ell} &=\sum_{\ell=0}^k \left[\binom{k-1}{\ell}+\binom{k-1}{\ell-1} \right] m_\ell \cdot (-\lambda)^{k-\ell}\\
 &=-\lambda \sum_{\ell=0}^{k-1}\binom{k-1}{\ell}m_\ell \cdot (-\lambda)^{k-1-\ell}+\sum_{\ell=0}^{k-1}\binom{k-1}{\ell}m_{\ell+1} \cdot (-\lambda)^{k-\ell-1}\\
 &=-\lambda \mu_{k-1}+\sum_{\ell=0}^{k-1}\binom{k-1}{\ell}\cdot (-\lambda)^{k-1-\ell} \cdot \lambda \sum_{s=0}^{\ell} \binom{\ell}{s}m_s.
\end{align*}
Setting $t:=k-1-\ell+s$ and noting that $$\binom{k-1}{\ell}\binom{\ell}{s}=\binom{k-1}{t}\binom{t}{s},$$
we can eliminate $\ell$ to obtain
\begin{align*}
\mu_k &=-\lambda \mu_{k-1}+\lambda\sum_{t=0}^{k-1} \sum_{s=0}^{t} \binom{k-1}{t}\binom{t}{s} (-\lambda)^{t-s} m_s\\
 &=-\lambda \mu_{k-1}+\lambda \sum_{t=0}^{k-1} \binom{k-1}{t} \mu_t\\
 &=\lambda \sum_{t=0}^{k-2} \binom{k-1}{t} \mu_t,
\end{align*}
as desired.
\end{proof}

\begin{proof}[Proof of Theorem~\ref{thm:primes-in-intervals-gaussian}]
Fix $k \in \mathbb{N}$.  Take $C=C(d), D=D(X), \tau=\tau(X), L(j), w=w(X)$ as in the proof of Theorem~\ref{thm: Poisson distribution}.  Theorem~\ref{thm:M} (with a union bound and keeping track of the dependence on $\mathcal{L}$) again gives
\begin{multline*}
\mathbb{P}_{f \in \mathcal{P}(d,H)}\left[\left|M_{\ell}(\rho_{f,X,L(j)}) -  \sum_{r=1}^\ell S(\ell,r)\left(\frac{L(j)\fS_f(w)}{\log(HX^d)}\right)^r\right| \geq \varepsilon \text{ for some $1 \leq \ell \leq k$ and $0 \leq j \leq \tau$}\right]\\
\ll_{k} \frac{\tau\varepsilon^{-2}(\mathcal{L}\log w)^{2k(d-1)}}{w^{1-o(1)}},
\end{multline*}
uniformly in $\varepsilon>0$.

Now pick $f \in \mathcal{P}(d,H)$ uniformly at random subject to the constraint $\fS_f(w) \neq 0$.  Arguing as in the proof of Theorem~\ref{thm: Poisson distribution} we conclude that with probability at least
$$1-O_{d,k} \left( \frac{\tau\varepsilon^{-2}(\mathcal{L}\log w)^{2k(d-1)}}{w^{1-o(1)}} \right),$$
we have
\begin{align}\label{eq:poisson-moment-single-f-next}
M_{\ell}\left(\rho_{f,X;\frac{\mathcal{L}\log(HX^d)}{\fS_f(w)}}\right)=\sum_{r=1}^\ell S(\ell,r)\mathcal{L}^r+O_{k}(\varepsilon+(\log D)/\tau)
\end{align}
for all $1 \leq \ell \leq k$.  Set $\tau=\tau(X):=\varepsilon^{-1}\log D$.  Consider the centered and rescaled distribution $\mathcal{L}^{-1/2}\left(\rho_{f,X;\frac{\mathcal{L}\log(HX^d)}{\fS_f(w)}}-\mathcal{L}\right)$.  By Lemma~\ref{lem:poisson-to-gaussian-conversion}, its $k$-th moment is
\begin{align*}
\frac{\sum_{\ell=0}^k \binom{k}{\ell} M_\ell\left(\rho_{f,X;\frac{\mathcal{L}\log(HX^d)}{\fS_f(w)}}\right) (-\mathcal{L})^{k-\ell}}{\mathcal{L}^{k/2}} &=\frac{\sum_{\ell=0}^k \binom{k}{\ell} m_\ell(\mathcal{L}) (-\mathcal{L})^{k-\ell}}{\mathcal{L}^{k/2}}+O_{k}\left(\sum_{\ell=1}^k \varepsilon \mathcal{L}^{k/2-\ell} \right)\\
&=C_k+O_k\left(\mathcal{L}^{-1/2}+\varepsilon \mathcal{L}^{k/2-1}\right).
\end{align*}
This error term is comfortably $o_k(1)$ if (say) $\varepsilon=\varepsilon(X):=\mathcal{L}^{-k/2}$.

Take a sequence of $X_i,H_i,w_i,f_i$ as in the statement of Theorem~\ref{thm:primes-in-intervals-gaussian}, and let $E_i$ denote the event that the last centered equation fails for $f_i$.  Then, by the assumption $\log \mathcal{L}(X)=o(\log w(X))$ and the growth rate of $X_i$, we have
\begin{align*}
\sum_{i=1}^\infty \Pr[E_i] &\ll_{d,k} \sum_{i=1}^\infty\frac{\tau(X_i)\varepsilon(X_i)^{-2}(\mathcal{L}(X_i)\log w_i)^{2k(d-1)}}{w_i^{1-o_k(1)}}\\
 &\ll_k \sum_{i=1}^\infty\frac{\varepsilon(X_i)^{-3}\mathcal{L}(X_i)^{2k(d-1)}}{w_i^{1-o_k(1)}}\\
&=\sum_{i=1}^\infty\frac{\mathcal{L}(X_i)^{3k/2+2k(d-1)}}{w_i^{1-o_k(1)}}\\
 &=\sum_{i=1}^\infty\frac{1}{w_i^{1-o_k(1)}}<\infty.
\end{align*}
Apply the Borel--Cantelli Lemma and take a union bound over $k$ as in the proof of Theorem~\ref{thm: Poisson distribution}.  This concludes the proof because the Gaussian distribution is determined by its moments.
\end{proof}

\section*{Acknowledgments}

We would like to thank Kevin Ford, Vivian Kuperberg, James Leng, James Maynard, Efthymios Sofos, and Joni Ter\"av\"ainen for helpful conversations and comments. We are grateful to Efthymios Sofos and Joni Ter\"av\"ainen for showing us the argument for higher moments in the polynomial range mentioned in \S\ref{subsec: related questions}.

The first author was supported in part by the National Science Foundation under grants DGE-203965 and DMS-2501336.  
The second author is partially supported by the NSF under grants DGE-2039656 and DMS-2502864. Any opinions, findings, and conclusions or recommendations expressed in this material are those of the author(s) and do not necessarily reflect the views of the National Science Foundation. The third author is supported by a Simons Junior Fellowship from the Simons Foundation.

\bibliography{biblio}
\bibliographystyle{plain}
\end{document}